%% file: sh-128.tex
\documentclass[a4paper]{amsart}

\input{macro}

\usepackage[initials,backrefs,msc-links]{amsrefs}

\title[Degrees of Genus-Two Split-Jacobian Loci]{Degrees of Genus-Two Split-Jacobian Loci and Humbert-Form Reconstruction}
 
\author{T. Shaska}
\address{Department of Mathematics and Statistics, Oakland University, Rochester, MI, 48309, USA}
\email{shaska@oakland.edu}


\begin{document}

\begin{abstract}
Let \(H_{n^2}\subset\cA_2\) be the Humbert surface of discriminant \(n^2\),  \(G_{n^2}\) be the Siegel modular form of level one with divisor \(H_{n^2}\), and  \(k(H_{n^2})\) be its weight.
Under the Torelli map, \(H_{n^2}\) is the closure of the locus \(\cL_n\subset\cM_2\) of genus-two curves admitting a maximal degree-\(n\) elliptic subcover, the zero locus of an irreducible weighted-homogeneous \(F_n\in\Z[J_2,J_4,J_6,J_{10}]\).
Let \(\nu(n)=\deg\bigl(X_1(n)\to X(1)\bigr)\).
We prove that the meromorphic form \(F_n(\tau)\) obtained from \(F_n\) has a pole of order exactly \(\nu(n)\) along the product locus, that \(G_{n^2}\) is a constant multiple of \(\chi_{10}^{\,\nu(n)}F_n(\tau)\), and that
\[
\degw F_n \;=\; k(H_{n^2})\;-\;10\,\nu(n)
\]
for every \(n\geq 2\), even or odd.
We determine the restriction of \(G_{n^2}\) to the product locus as a product of modular polynomials, its leading Fourier--Jacobi coefficient as a product of theta functions over the points of exact order \(n\), and we characterize \(G_{n^2}\) up to scalar by vanishing on a single torsion divisor.
These data convert the computation of \(G_{n^2}\) into a linear problem of the prescribed size, which we carry out for \(n=5\), producing the exact Humbert form \(G_{25}\) in the generators $\psi_4,\psi_6,\chi_{10},\chi_{12}$ and the equation \(F_5\) over $\Q$.
\end{abstract}

\maketitle

\setcounter{tocdepth}{2}

\input{main}

\input{sec-7}

\bibliographystyle{amsalpha}
\bibliography{sh-128}

\end{document}

%% file: macro.tex
\usepackage{amsmath,amssymb,amsfonts,amsthm}
\usepackage{mathtools}

\usepackage{tikz}
\usepackage{tikz-cd}
\usetikzlibrary{calc,intersections}

\usepackage[
  colorlinks=true,
  linkcolor=blue,
  citecolor=blue,
  urlcolor=blue
]{hyperref}

\usepackage[nameinlink,noabbrev]{cleveref}

\numberwithin{equation}{section}

\newtheorem{thm}{Theorem}[section]
\newtheorem{prop}[thm]{Proposition}
\newtheorem{lem}[thm]{Lemma}
\newtheorem{cor}[thm]{Corollary}

\newtheorem{rem}[thm]{Remark}

\newtheorem*{thm*}{Theorem}
\newtheorem*{cor*}{Corollary}

\theoremstyle{definition}
\newtheorem{defn}[thm]{Definition}

\crefname{equation}{Eq.}{Eqs.}
\Crefname{equation}{Eq.}{Eqs.}
\crefformat{equation}{Eq.~#2(#1)#3}
\Crefformat{equation}{Eq.~#2(#1)#3}

\crefname{thm}{Thm.}{Thms.}
\Crefname{thm}{Thm.}{Thms.}
\crefname{prop}{Prop.}{Props.}
\Crefname{prop}{Prop.}{Props.}
\crefname{lem}{Lem.}{Lems.}
\Crefname{lem}{Lem.}{Lems.}
\crefname{cor}{Cor.}{Cors.}
\Crefname{cor}{Cor.}{Cors.}
\crefname{rem}{Rem.}{Rems.}
\Crefname{rem}{Rem.}{Rems.}
\crefname{defn}{Def.}{Defs.}
\Crefname{defn}{Def.}{Defs.}
\crefname{exa}{Ex.}{Exs.}
\Crefname{exa}{Ex.}{Exs.}

\crefname{section}{Sec.}{Secs.}
\Crefname{section}{Sec.}{Secs.}
\crefname{subsection}{Sec.}{Secs.}
\Crefname{subsection}{Sec.}{Secs.}

\crefname{table}{Tab.}{Tabs.}
\Crefname{table}{Tab.}{Tabs.}
\crefname{figure}{Fig.}{Figs.}
\Crefname{figure}{Fig.}{Figs.}

\newcommand{\bA}{\mathbb{A}}
\newcommand{\bH}{\mathbb{H}}
\newcommand{\F}{\mathbb{F}}
\newcommand{\Gm}{\mathbb{G}_m}

\DeclareMathOperator{\Sp}{Sp}
\DeclareMathOperator{\Div}{div}
\DeclareMathOperator{\ord}{ord}

\newcommand{\degw}{\deg_w}

\newcommand{\wphi}{\tilde{\phi}}

\newcommand{\bP}{\mathbb P}

\newcommand{\Q}{\mathbb{Q}}
\newcommand{\R}{\mathbb{R}}

\newcommand{\C}{\mathbb{C}}

\newcommand{\cP}{\mathcal{P}}

\newcommand{\Z}{\mathbb{Z}}

\DeclareMathOperator{\Frac}{Frac}
\DeclareMathOperator{\Spec}{Spec}
\DeclareMathOperator{\Pic}{Pic}
\DeclareMathOperator{\End}{End}

\DeclareMathOperator{\diag}{diag}

\DeclareMathOperator{\Aut}{Aut}

\DeclareMathOperator{\SL}{SL}

\DeclareMathOperator{\chara}{char}

\newcommand{\Jac}{\operatorname{Jac}}

\newcommand{\cO}{\mathcal O}
\newcommand{\cA}{\mathcal A}
\newcommand{\cE}{\mathcal E}
\newcommand{\cH}{\mathcal H}
\newcommand{\cL}{\mathcal L}
\newcommand{\cR}{\mathcal R}

\newcommand{\cM}{\mathcal{M}}

\DeclareMathOperator{\rank}{rank}

%% file: main.tex
\section{Introduction}
\label{sec-1}

For an integer \(n\geq 2\), let \(\cL_n\subset\cM_2\) denote the locus of genus two curves admitting a maximal degree \(n\) elliptic subcover.   The Igusa invariants identify \(\cM_2\) with the open subvariety \(\{J_{10}\neq 0\}\) of \(\bP(2,4,6,10)\), and there \(\cL_n\) is the zero set of a single irreducible weighted-homogeneous polynomial \(F_n\in\Z[J_2,J_4,J_6,J_{10}]\).
The first systematic computation of these equations was carried out in \cite{2001-0}: the loci were parametrized by invariants of a group action on the parameters of the defining families, the dihedral invariants, and the equations were obtained by elimination with resultants. This method produced the equations of \(\cL_n\) for $n=2, 3$  in \cite{2000-2}, \cite{2004-2} by this author, and later $n=5$ in \cite{MSV09}.    The case of $n$ even is slightly different and the equation of  \(F_4\)    was computed  in \cite{BD11}  by interpolation.  Kumar in \cite{Kum15}  constructed dominant families of split Jacobians for every \(n\le 11\).
The weighted degrees of such polynomials are  \(\degw F_2=30\),  \(\degw F_3=80\),   \(\degw F_4=180\), and \(\degw F_5=480\).  
A formula for \(\degw F_n\), a description of the support, and a bound on the size of the linear problem remained unavailable before this paper.


Let \(\cA_2\) be the moduli space of principally polarized abelian surfaces and let \(\iota:\cM_2\to\cA_2\) be the Torelli map, sending a curve to its Jacobian with the theta polarization. The closure of \(\iota(\cL_n)\) in \(\cA_2\) is the Humbert surface \(H_{n^2}\). An exact Humbert form \(G_{n^2}\) is a Siegel modular form of level one with divisor exactly \(H_{n^2}\); its weight \(k(H_{n^2})\) is computable for every \(n\) from van der Geer's recursion. Substituting Igusa's invariants in terms of Siegel modular forms  into \(F_n\) produces a meromorphic Siegel modular form \(F_n(\tau)\) whose only pole is along the product locus \(H_1=\{\chi_{10}=0\}\).
Let us  denote  
\(
\deg\bigl(X_1(n)\to X(1)\bigr), 
\)
by
\[
\nu(n)=\deg\bigl(X_1(n)\to X(1)\bigr)=
\begin{cases}
3, & n=2,\\
\tfrac{1}{2}\,\varphi(n)\,\psi(n), & n\geq 3,
\end{cases}
\]
where \(\varphi\) is Euler's function and \(\psi\) is the Dedekind psi function.
The main result of this paper is that the pole of \(F_n(\tau)\) along \(H_1\) has order exactly \(\nu(n)\). In other words,  for every \(n\geq 2\) there is \(c\in\C^\times\) such that
\(
G_{n^2}=c\,\chi_{10}^{\,\nu(n)}F_n(\tau).
\)
The essential point is the identification of the exponent with a torsion count: \(\nu(n)\) is the number of pairs \((E,\pm x)\), where \(x\) is a point of exact order \(n\).
 Since \(G_{n^2}\) has weight \(k(H_{n^2})\), while \(\chi_{10}\) has weight \(10\) and \(F_n(\tau)\) has weight \(\degw F_n\), comparison of weights gives
\[
\operatorname{wt}(G_{n^2})  =k(H_{n^2})   =10\,\nu(n)+\degw F_n,
\]
and hence
\(
\degw F_n=k(H_{n^2})-10\,\nu(n)
\)
for every \(n\geq 2\), even or odd.

The main idea of the proof is as follows. 
Let  \(e_n\) be the pole order of \(F_n(\tau)\) along \(H_1\). A comparison of divisor classes in \(\Pic_{\Q}(\cA_2)=\Q\lambda\) gives \(\degw F_n=k(H_{n^2})-10\,e_n\) and \(G_{n^2}=c\,\chi_{10}^{e_n}F_n(\tau)\) (\cref{thm:closed,cor:exact-form}). Everything then reduces to the determination of \(e_n\). However,   \(e_n\) is a boundary invariant, so it is the index of the first non-vanishing Fourier-Jacobi coefficient of \(G_{n^2}\) (\cref{thm:fj-order}). Humbert's singular relations force every Fourier-Jacobi coefficient to vanish at every torsion point of exact order \(n\), and counting zeros of Jacobi forms gives \(e_n\geq\nu(n)\) (\cref{prop:lower-bound}). For the reverse inequality, the restriction \(g_n=G_{n^2}|_{z=0}\) satisfies \(e_n\leq\ord_{q_2}(g_n)\), and \(g_n/(\Delta\Delta')^{k/12}\) is a polynomial in the two \(j\)-invariants whose zero divisor is the intersection cycle \(H_{n^2}\cdot H_1\), a sum of isogeny curves with forced multiplicities. An elementary identity for the Humbert weights, \(k(H_{n^2})=12\,\nu(n)+12\,\widetilde{R}(n)\) for \(n\geq 3\), with \(\widetilde{R}(n)\) an explicit divisor sum (\cref{thm:weight-identity}), converts the degree count into \(e_n\leq\nu(n)\).

Equality throughout determines the boundary data of \(G_{n^2}\). 
The restriction \(g_n\) is a product of classical modular polynomials
(\cref{cor:boundary-product}), and, for \(n\geq 3\), the leading Fourier-Jacobi coefficient is
a product of theta functions over the torsion points of order \(n\) (\cref{prop:theta}).
These data do not determine \(G_{n^2}\), but a single linear condition along a torsion divisor does. The resulting procedure computes \(F_n\) by linear algebra in spaces of modular forms rather than by elimination.
This is a different computational approach to the equations \(F_n\) other than  elimination methods of \cite{2000-2, 2001-0, 2004-2}.
 
The exponent \(\nu(n)\) counts pairs \((E,\pm x)\) with \(x\) of exact order \(n\), so the branches are indexed by \(X_1(n)\). The count \(\psi(n)\) of cyclic subgroups of order \(n\), indexed by \(X_0(n)\), agrees with \(\nu(n)\) exactly when \(\varphi(n)\leq 2\), that is, for \(n\in\{2,3,4,6\}\); the two differ first at \(n=5\), the smallest \(n\) with \(\varphi(n)=4\).

The procedure is carried out at \(n=5\). In \(\Z[J_2,J_4,J_6,J_{10}]\) the equation \(F_5\) has \(76713\) terms among the \(82189\) weighted-homogeneous candidates of degree \(480\), recorded in \cref{sec-7} and accompanying this article as an ancillary file. 
The number of unknowns is prescribed by the degree formula; the height \( h (G_{n^2})\)
of the primitive modular coefficient vector is bounded only through the interpolation points, by \cref{lem:hadamard}, whose bound at \(n=5\) overshoots the actual height by five orders of magnitude. The former fixes the cost of one solve, which at \(n=7\) requires about two tebibytes, whereas the latter fixes the number of primes and is known only after the modular coefficient vector has been reconstructed.

\Cref{sec-2} collects the preliminaries. \Cref{sec-3} proves the closed formula, the Fourier-Jacobi characterization, and the lower bound. \Cref{sec-4} proves the weight identity. \Cref{sec-5} proves the factorization and the degree formula. \Cref{sec-6} records the boundary data in closed form and the characterization of \(G_{n^2}\) by a single linear condition. \Cref{sec-7} records the procedure that computes \(F_n\) from the boundary data and the computation of \(F_5\), and the arithmetic size of the problem.
\section{Modular and geometric preliminaries}
\label{sec-2}

We write \(\bP(2,4,6,10)\) for the weighted projective space with weights the degrees of the invariants \(J_2,J_4,J_6,J_{10}\), that is, the quotient of \(\bA^4\setminus\{0\}\) by the \(\Gm\)-action \(t\cdot(x_0,x_1,x_2,x_3)=(t^2x_0,t^4x_1,t^6x_2,t^{10}x_3)\), and \(\degw\) for the weighted degree of a weighted-homogeneous polynomial.
A genus two curve \(C\) defined over a field \(K\) of characteristic \(\chara(K)\neq 2\) is a curve with affine equation
\begin{equation}
\label{affine:C}
y^2 = a_6 x^6 + a_5 x^5 + \cdots + a_1 x + a_0.
\end{equation}
Since \(\chara(K)\neq 2\), the ring of even-degree \(\SL_2\)-invariants of binary sextics is generated by the \textbf{Igusa invariants} \(J_2,J_4,J_6,J_{10}\), and these identify \(\cM_2\) with the open subvariety \(\{J_{10}\neq 0\}\) of \(\bP(2,4,6,10)\). We write \(I(C)=[J_2:J_4:J_6:J_{10}]\) for the moduli point of \(C\).

\begin{defn}
\label{defn:subcover}
Let \(C\) be a genus-two curve over \(\C\). An \textbf{elliptic subcover} of degree \(n\) is a finite covering \(\pi:C\to E\) onto an elliptic curve with \(\deg\pi=n\). It is \textbf{maximal} if it does not factor as \(C\to E'\to E\) with \(E'\to E\) an isogeny of degree larger than one. For \(n\geq 2\) set
\[
\cL_n=\bigl\{\,[C]\in\cM_2 \;:\; C\ \text{admits a maximal degree-}n\ \text{elliptic subcover}\,\bigr\}.
\]
\end{defn}

\subsection{Modular curves}
\label{subsec:modular-curves}

Let \(X(1)\) be the modular curve of elliptic curves, the compactification of \(\bH/\SL_2(\Z)\), with coordinate the absolute invariant \(j\). For \(n\geq 2\), the curve \(X_0(n)\) parametrizes pairs \((E,G)\) of an elliptic curve with a cyclic subgroup of order \(n\), and the curve \(X_1(n)\) parametrizes pairs \((E,\pm x)\) of an elliptic curve with a point \(x\) of exact order \(n\), taken up to inversion since \(-1\in\Aut(E)\). The degrees of the natural coverings to \(X(1)\) are
\[
\deg\bigl(X_0(n)\to X(1)\bigr)=\psi(n),\qquad 
\deg\bigl(X_1(n)\to X(1)\bigr)=\nu(n),
\]
where \(\psi(n)=n\prod_{p\mid n}(1+\tfrac{1}{p})\) is the Dedekind psi function and
\[
\nu(n)=
\begin{cases}
3, & n=2,\\
\tfrac{1}{2}\,\varphi(n)\,\psi(n), & n\geq 3,
\end{cases}
\]
with \(\varphi\) Euler's function. Indeed
\[
[\SL_2(\Z):\Gamma_1(n)]=n^2\prod_{p\mid n}(1-p^{-2})=\varphi(n)\psi(n),
\]
which is the number of points of exact order \(n\) on an elliptic curve, and \(-I\in\Gamma_1(n)\) if and only if \(n\leq 2\). One has \(\nu(n)=\psi(n)\) if and only if \(\varphi(n)\leq 2\), that is, for \(n\in\{2,3,4,6\}\).

For \(M\geq 1\) let \(T_M\subset X(1)\times X(1)\) be the modular curve of cyclic \(M\)-isogenies in the coordinates \((j_1,j_2)\): the image of \(\tau_1\mapsto(j(\tau_1),j(P\tau_1))\) for any primitive integer matrix \(P\) with \(\det P=M\). It is an irreducible plane curve of degree \(\psi(M)\) in each variable, cut out by the 
\textbf{classical modular polynomial} \(\Phi_M(j_1,j_2)\), and \(T_M\neq T_{M'}\) for \(M\neq M'\).

\subsection{Modular forms}

Let \(\bH_2\) be the Siegel upper half-space of degree two, with coordinates
\[
\tau=\begin{pmatrix}\tau_1 & z\\ z & \tau_2\end{pmatrix},\qquad q_1=e^{2\pi i\tau_1},\qquad q_2=e^{2\pi i\tau_2},
\]
and let \(\cA_2=\bH_2/\Sp_4(\Z)\) be the moduli space of principally polarized abelian surfaces, \(\iota:\cM_2\to\cA_2\) the Torelli map, and \(\lambda\) the Hodge line bundle. Every principally polarized abelian surface is either the Jacobian of a smooth genus-two curve or a product of two elliptic curves with the product polarization (see \cite{sh-93}), so \(\iota\) is injective with complement the product locus. We use only one fact about the Picard group: \(\Pic_{\Q}(\cA_2)=\Q\lambda\), and a meromorphic Siegel modular form of weight \(w\) has divisor class \(w\lambda\).

Let \(M_k=M_k(\Sp_4(\Z))\) be the space of scalar Siegel modular forms of weight \(k\) and level one, and let \(\cE_w=M_w(\SL_2(\Z))\) be the space of elliptic modular forms of weight \(w\) and level one. Igusa's structure theorem states that the even-weight ring
\[
M^{\mathrm{even}}:=\bigoplus_{k\ \mathrm{even}} M_k=\C[\psi_4,\psi_6,\chi_{10},\chi_{12}]
\]
is a free polynomial algebra, where \(\psi_4,\psi_6\) are the Eisenstein series of weights \(4,6\) and \(\chi_{10},\chi_{12}\) are cusp forms of weights \(10,12\); the full ring is generated over it by the odd cusp form \(\chi_{35}\); see \cite{Igusa62,Igusa67}.

Igusa expressed the invariants of the sextic in \cref{affine:C} through these generators, for the period matrix \(\tau\) of \(\Jac(C)\): up to nonzero rational constants, \(I_2=\chi_{12}/\chi_{10}\), \(I_4=\psi_4\), \(I_{10}=\chi_{10}\), and \(I_6\) is a linear combination of \(\psi_6\) and \(\psi_4\chi_{12}/\chi_{10}\); see\cite{Igusa67}.   The constants in the present normalization being written out in \cite{2016-3}. 
The only polar divisor of these coordinates is the product locus \(H_1=\{\chi_{10}=0\}\), so \(\iota(\cM_2)=\cA_2\setminus H_1\), and \(\Div(\chi_{10})=H_1\) with \([H_1]=10\lambda\).

Every \(F\in M_k\) has a Fourier-Jacobi expansion
\[
F(\tau)=\sum_{m\geq 0}\phi_m(\tau_1,z)\,q_2^m,
\]
in which \(\phi_m\) is a Jacobi form of weight \(k\) and index \(m\) for the Jacobi group \(\SL_2(\Z)\ltimes\Z^2\); see \cite{EichlerZagier}. We write \(\ord_{q_2}(F)\) for the smallest \(m\) with \(\phi_m\neq 0\). The constant coefficient \(\phi_0\) is independent of \(z\) and equals the image of \(F\) under the Siegel operator; which sends \(\psi_4,\psi_6\) to \(E_4,E_6\) and annihilates \(\chi_{10}\) and \(\chi_{12}\). The index-one Fourier-Jacobi coefficients of \(\chi_{10}\) and \(\chi_{12}\) are nonzero constant multiples of the Jacobi cusp forms \(\phi_{10,1}\) and \(\phi_{12,1}\) of Eichler--Zagier; indeed \(\chi_{10}\) and \(\chi_{12}\) are the Maass lifts of \(\phi_{10,1}\) and \(\phi_{12,1}\); see \cite[Ch.~6]{EichlerZagier}. Hence
\[
\ord_{q_2}(\psi_4)=\ord_{q_2}(\psi_6)=0,\quad    
\ord_{q_2}(\chi_{10})=\ord_{q_2}(\chi_{12})=1.
\]
A \textbf{Jacobi form} of weight \(k\) and index \(m\) for \(\SL_2(\Z)\ltimes\Z^2\) is a holomorphic function \(\phi:\bH\times\C\to\C\) satisfying the modular law for every
\(\begin{pmatrix}a&b\\ c&d\end{pmatrix}\in\SL_2(\Z)\),
\[
\phi\Bigl(\frac{a\tau_1+b}{c\tau_1+d},\,\frac{z}{c\tau_1+d}\Bigr)=(c\tau_1+d)^k\,e^{2\pi imcz^2/(c\tau_1+d)}\,\phi(\tau_1,z),
\]
the elliptic law
\[
\phi(\tau_1,z+\ell\tau_1+\ell')=e^{-2\pi im(\ell^2\tau_1+2\ell z)}\,\phi(\tau_1,z),\qquad \ell,\ell'\in\Z,
\]
and admitting a Fourier expansion
\[
\phi=\sum_{N,r}c(N,r)\,q_1^N\zeta^r
\]
with \(\zeta=e^{2\pi iz}\) and \(c(N,r)=0\) unless \(4Nm\geq r^2\); see \cite{EichlerZagier}. It is a cusp form if \(c(N,r)=0\) unless \(4Nm>r^2\).

We use two further standard facts from \cite{EichlerZagier}. First, for fixed \(\tau_1\) the function \(z\mapsto\phi(\tau_1,z)\), if not identically zero, has exactly \(2m\) zeros in a fundamental domain for \(\C/\Lambda_{\tau_1}\), \(\Lambda_{\tau_1}=\Z\tau_1+\Z\), counted with multiplicity \cite[Thm.~1.2]{EichlerZagier}. Second, for even \(k\) one has \(\phi(\tau_1,-z)=\phi(\tau_1,z)\).

We record one consequence for the weak Jacobi generator \(\wphi_{-2,1}\) of Eichler and Zagier \cite[Thm.~9.3]{EichlerZagier}. The restriction \(\wphi_{-2,1}(\tau_1,0)\) is a holomorphic modular form of weight \(-2\), hence zero; the weight being even, the vanishing order at \(z=0\) is at least \(2\), and an index-one Jacobi form has exactly two zeros \cite[Thm.~1.2]{EichlerZagier}. Hence the divisor of \(z\mapsto\wphi_{-2,1}(\tau_1,z)\) on \(\C/\Lambda_{\tau_1}\) is \(2\cdot(0)\).

\begin{lem}
\label{lem:2torsion-parity}
Let \(\phi\) be a Jacobi form of even weight and index \(m\in\Z\), and let \(z_0=(\ell\tau_1+\ell')/2\) with \(\ell,\ell'\in\{0,1\}\), \((\ell,\ell')\neq(0,0)\), be a nontrivial 2-torsion point. Then the vanishing order of \(z\mapsto\phi(\tau_1,z)\) at \(z_0\) is even.
\end{lem}

\begin{proof}
Write \(2z_0=\ell\tau_1+\ell'\) and let \(w\) be near \(0\). Evenness in \(z\), which holds
since the weight is even, gives
\(
\phi(\tau_1,z_0-w)=\phi(\tau_1,-z_0+w).
\)
The argument on the right may be rewritten using \(-z_0=(z_0+w)-2z_0-w\), that is,
\[
-z_0+w=(z_0+w)-2z_0=(z_0+w)-(\ell\tau_1+\ell'),
\]
so it differs from \(z_0+w\) by the lattice translation \(-\ell\tau_1-\ell'\). The elliptic
law with \((\ell,\ell')\) replaced by \((-\ell,-\ell')\), applied at \(z=z_0+w\), therefore gives
\[
\phi(\tau_1,z_0-w)=u(\tau_1,w)\,\phi(\tau_1,z_0+w),
\qquad
u(\tau_1,w)=e^{-2\pi im(\ell^2\tau_1-2\ell(z_0+w))},
\]
with \(u\) nowhere vanishing. Substituting \(2z_0=\ell\tau_1+\ell'\) at \(w=0\) cancels the
\(\tau_1\)-terms and gives \(u(\tau_1,0)=e^{2\pi im\ell\ell'}=1\), since \(m\ell\ell'\in\Z\). If
\(t\) denotes the vanishing order at \(z_0\), comparing the coefficients of \(w^t\) on both
sides yields \((-1)^t=1\).
\end{proof}

\begin{lem}
\label{lem:algindep}
Functions \(E_4\), \(E_6\), \(\phi_{10,1}\), \(\phi_{12,1}\) are algebraically independent over \(\C\).
\end{lem}

\begin{proof}
By the structure theorem for weak Jacobi forms of even weight
\cite[Thm.~9.3]{EichlerZagier}, the ring of such forms is the free polynomial algebra
\(\C[E_4,E_6,\wphi_{-2,1},\wphi_{0,1}]\), and \(\phi_{10,1}\), \(\phi_{12,1}\) are nonzero
constant multiples of \(\Delta\wphi_{-2,1}\), \(\Delta\wphi_{0,1}\), where
\(\Delta=(E_4^3-E_6^2)/1728\). The four generators are therefore algebraically independent,
so the field \(\C(E_4,E_6,\wphi_{-2,1},\wphi_{0,1})\) has transcendence degree four over \(\C\).
Since \(\Delta\) lies in \(\C[E_4,E_6]\) and is nonzero, it is invertible in that field, whence
\[
\C\bigl(E_4,E_6,\phi_{10,1},\phi_{12,1}\bigr)=\C\bigl(E_4,E_6,\wphi_{-2,1},\wphi_{0,1}\bigr).
\]
Four elements generating a field of transcendence degree four are algebraically independent.
\end{proof}

\subsection{Humbert surfaces}

For \(D\equiv 0,1\pmod 4\), \(D>0\), let \(H_D\subset\cA_2\) be the Humbert surface of discriminant \(D\); it is an irreducible divisor, by Humbert's theorem \cite{Humbert1899}, see \cite[Ch.~IX]{HvdG}. 
The preimage of \(H_D\) in \(\bH_2\) is the union, over all primitive quintuples \((a,b,c,d,e)\in\Z^5\) with \(b^2-4ac-4de=D\), of the singular relations
\begin{equation}
\label{eq:singular-relation}
a\tau_1+bz+c\tau_2+d(z^2-\tau_1\tau_2)+e=0,
\end{equation}
and \(\Sp_4(\Z)\) acts transitively on the primitive quintuples of a fixed discriminant.
The quintuple \((0,1,0,0,0)\) gives the relation \(z=0\), so \(H_1\) is the product locus.

For every \(n\geq 2\), with no parity assumption on \(n\), a curve \(C\) lies in \(\cL_n\) if and only if \(\iota(C)\in H_{n^2}\), and
\[
\iota(\cL_n)=H_{n^2}\cap\iota(\cM_2)
\]
is Zariski open and dense in \(H_{n^2}\); see  \cite{MSV09}. In particular \(\cL_n\) is a closed irreducible subvariety of \(\cM_2\) of dimension two for every \(n\geq 2\), even or odd: it is the preimage of \(H_{n^2}\) under \(\iota\), and it is open and dense in the irreducible surface \(H_{n^2}\). 
The complement \(H_{n^2}\setminus\iota(\cL_n)=H_{n^2}\cap H_1\) is a curve of products of isogenous elliptic curves; its image in the coordinates \((j_1,j_2)\) is a union of the modular curves \(T_M\) of \cref{subsec:modular-curves}; see \cite{Kani97} for details.

The closure of \(\cL_n\) in \(\bP(2,4,6,10)\) is an irreducible surface, hence a prime Weil divisor. The divisor class group of \(\bP(2,4,6,10)\) is generated by \(\cO(1)\), so this closure is the zero set of a single irreducible weighted-homogeneous polynomial, unique up to a nonzero scalar. The locus is defined over \(\Q\) \cite{Kani97}, so this polynomial may be taken primitive. It is unique up to multiplication by a sign.
We denote it by  \(F_n\in\Z[J_2,J_4,J_6,J_{10}]\) and will call it the \textbf{equation} of \(\cL_n\). 

An \textbf{exact Humbert form} for \(D\) is a scalar Siegel modular form \(G_D\) of level one with \(\Div(G_D)=H_D\); its weight is denoted \(k(H_D)\). Such a form exists for every \(D\) \cite{HvdG,Gruenewald}, and it is unique up to a nonzero scalar, since the ratio of two forms with the same divisor is holomorphic of weight zero together with its inverse. For \(D=1\) one has \(G_1=\chi_{10}\) and \(k(H_1)=10\). The weight \(k(H_D)\) is even for every \(D\): the element \(\diag(1,-1,1,-1)\in\Sp_4(\Z)\) maps \((\tau_1,z,\tau_2)\) to \((\tau_1,-z,\tau_2)\) with automorphy factor \((-1)^k\), so a form of odd weight vanishes on \(\{z=0\}\), its divisor contains \(H_1\), and this contradicts \(\Div(G_D)=H_D\) for \(D>1\).

By van der Geer's theorem on the Humbert generating series \cite{HvdG}, the weights satisfy
\[
\sum_{m\mid n} v(m^2)\,k(H_{m^2})=\tfrac{1}{2}\,a_{n^2},\qquad v(m^2)=\begin{cases}\tfrac12,&m\in\{1,2\},\\ 1,&m\geq 3,\end{cases}
\]
where \(\sigma_1(N)=\sum_{d\mid N}d\) and \(a_D\) is the \(D\)-th Fourier coefficient of \(120\,\cH_{5/2}\), the Cohen--Eisenstein series of weight \(5/2\) \cite{Cohen75}, given by Siegel's formula
\[
a_D=24\sum_{\substack{x\in\Z\\ (D-x^2)/4\in\Z_{>0}}}\sigma_1\Bigl(\frac{D-x^2}{4}\Bigr)+
\begin{cases}
12D-2,&D\text{ a square},\\
0,&\text{otherwise}.
\end{cases}
\]
Together with \(k(H_1)=10\), the recursion determines \(k(H_{n^2})\) for every \(n\). 

\section{The Humbert form along the product locus}
\label{sec-3}


Let \(F_n\in\Z[J_2,J_4,J_6,J_{10}]\) be the irreducible equation of \(\cL_n\) and let \(F_n(\tau)\) denote the meromorphic Siegel modular form of weight \(\degw F_n\) obtained by the substitution of \cref{sec-2}.  Define the order of the polar divisor as 
\begin{equation}
\label{defn:en}
e_n:=-\ord_{H_1}\bigl(F_n(\tau)\bigr)\in\Z.
\end{equation}

\begin{thm}
\label{thm:closed}
For every \(n\geq 2\),
\[
\degw F_n \;=\; k(H_{n^2})\;-\;10\,e_n.
\]
\end{thm}

\begin{proof}
We prove
\(
\Div\bigl(F_n(\tau)\bigr)=H_{n^2}-e_nH_1
\)
and then compare divisor classes.

On \(\cA_2\setminus H_1=\iota(\cM_2)\), the form \(F_n(\tau)\) is
holomorphic and vanishes exactly on \(\iota(\cL_n)\), because \(F_n\)
cuts out \(\cL_n\) in \(\cM_2\). Since the closure of \(\iota(\cL_n)\)
in \(\cA_2\) is \(H_{n^2}\) by \cref{sec-2}, every component of
\(\Div(F_n(\tau))\) other than \(H_1\) is \(H_{n^2}\). Thus
\[
\Div(F_n(\tau))=mH_{n^2}+aH_1
\]
for some \(m\geq 1\) and \(a\in\Z\), and \(a=-e_n\) by \cref{defn:en}.

It remains to show that \(m=1\). Since \(\iota\) is an open immersion
onto \(\cA_2\setminus H_1\), it suffices to show that \(F_n\) cuts out
\(\cL_n\) with multiplicity one in \(\bP(2,4,6,10)\). All four weights
are even, so the chart \(\{J_2\neq 0\}\) is the affine space with
coordinates \(J_4/J_2^2\), \(J_6/J_2^3\), \(J_{10}/J_2^5\). Since
\(F_n\) is irreducible of weighted degree larger than two, it is not a
constant multiple of \(J_2\), so \(J_2\nmid F_n\) and the closure of
\(\cL_n\) meets the chart; being dense in its closure, \(\cL_n\) meets
it as well.

On the chart, the dehomogenization \(F_n/J_2^{\degw F_n/2}\) is
irreducible. Indeed, rehomogenizing a factorization into two
nonconstant factors gives \(J_2^{\,e}F_n=G_1G_2\) with \(e\geq 0\) and
\(G_1,G_2\) weighted-homogeneous of positive degree; since \(F_n\) is
irreducible and \(J_2\nmid F_n\), one \(G_i\) is a power of \(J_2\),
whose dehomogenization is constant, a contradiction. The
dehomogenization therefore generates a radical ideal, \(F_n\) vanishes
with multiplicity one along \(\cL_n\), and \(m=1\).

Finally, a meromorphic Siegel modular form of weight \(w\) has divisor
class \(w\lambda\), while \([H_{n^2}]=k(H_{n^2})\lambda\) and
\([H_1]=10\lambda\), because \(H_{n^2}=\Div(G_{n^2})\) and
\(H_1=\Div(\chi_{10})\). Comparing classes in
\(\Pic_{\Q}(\cA_2)=\Q\lambda\),
\[
\degw F_n\cdot\lambda
  = [H_{n^2}]-e_n[H_1]
  = \bigl(k(H_{n^2})-10e_n\bigr)\lambda,
\]
and \(\degw F_n=k(H_{n^2})-10e_n\).
\end{proof}

\begin{cor}
\label{cor:exact-form}
There is \(c\in\C^\times\) with
\(
G_{n^2}=c\,\chi_{10}^{\,e_n}\,F_n(\tau),
\)
and \(G_{n^2}\) is unique up to a nonzero scalar.
\end{cor}

\begin{proof}
By the proof of \cref{thm:closed}, \(\Div(\chi_{10}^{\,e_n}F_n(\tau))=e_n H_1+H_{n^2}-e_nH_1=H_{n^2}\), and the weight is \(10e_n+\degw F_n=k(H_{n^2})\). The ratio of two meromorphic forms with equal divisor is holomorphic of weight zero together with its inverse, hence a nonzero constant.
\end{proof}

\begin{lem}
\label{lem:igusa-relations}
Set \(T:=\chi_{12}/\chi_{10}\). There is an isomorphism of \(\C\)-algebras
\[
\Theta\colon\C[J_2,J_4,J_6,J_{10}]\;\xrightarrow{\ \sim\ }\;\C[T,\psi_4,\psi_6,\chi_{10}]
\]
under which
\[
J_2\mapsto a\,T,\quad
J_4\mapsto b_1\psi_4+b_2T^2,\quad
J_6\mapsto c_1\psi_6+c_2\psi_4T+c_3T^3,\quad
J_{10}\mapsto d\,\chi_{10},
\]
with \(a,b_1,c_1,d\in\Q^{\times}\) and \(b_2,c_2,c_3\in\Q\). For every \(n\geq 2\) one has
\(\Theta(F_n)=F_n(\tau)\).
\end{lem}

\begin{proof}
Igusa's relations express the Igusa--Clebsch invariants through the generators
\cite{Igusa67}: there are \(\alpha,\beta,\delta,\gamma_1\in\Q^{\times}\) and
\(\gamma_2\in\Q\) with
\[
I_2=\alpha\,\frac{\chi_{12}}{\chi_{10}},\qquad
I_4=\beta\,\psi_4,\qquad
I_6=\gamma_1\psi_6+\gamma_2\psi_4\frac{\chi_{12}}{\chi_{10}},\qquad
I_{10}=\delta\,\chi_{10}.
\]
Composing with the triangular substitution between \((I_2,I_4,I_6,I_{10})\) and
\((J_2,J_4,J_6,J_{10})\) of \cref{sec-2} gives the displayed assignments, the leading
constants \(a,b_1,c_1,d\) being nonzero rational multiples of \(\alpha,\beta,\gamma_1,\delta\).
The four elements \(T,\psi_4,\psi_6,\chi_{10}\) are algebraically independent, since
\(\psi_4,\psi_6,\chi_{10},\chi_{12}\) are and \(\chi_{12}=T\chi_{10}\); the assignments are
triangular in the order \(T,\psi_4,\psi_6,\chi_{10}\) with \(a,b_1,c_1,d\neq0\), so the map is
an isomorphism onto \(\C[T,\psi_4,\psi_6,\chi_{10}]\).
\end{proof}

\begin{thm}
\label{thm:fj-order}
Let \(n\geq 2\),  \(e_n\) be the polar order of \cref{defn:en}, and  \(G_{n^2}\)
be the exact Humbert form of discriminant \(n^2\). Then \(e_n=\ord_{q_2}(G_{n^2})\).
\end{thm}

\begin{proof}
By \cref{cor:exact-form} and \cref{lem:igusa-relations},
\(G_{n^2}=c\,\chi_{10}^{\,e_n}\,Q\) with \(Q=\Theta(F_n)\in\C[T,\psi_4,\psi_6,\chi_{10}]\)
and \(c\in\C^\times\). Here \(\chi_{10}\nmid Q\), since \(J_{10}\nmid F_n\) and
\(J_{10}\mapsto\chi_{10}\) up to a nonzero constant.

Write \(Q=\sum q_{i,j,a,b}\,T^i\psi_4^a\psi_6^b\chi_{10}^j\). Since \(T=\chi_{12}/\chi_{10}\),
\[
\chi_{10}^{\,e_n}Q=\sum q_{i,j,a,b}\,\psi_4^a\psi_6^b\chi_{10}^{\,e_n+j-i}\chi_{12}^{\,i},
\]
an identity in \(\C[\psi_4,\psi_6,\chi_{12}][\chi_{10}^{-1}]\). The map
\((i,j)\mapsto(e_n+j-i,\,i)\) is injective, so distinct quadruples \((i,j,a,b)\) give distinct
terms on the right and no cancellation occurs. Hence the terms of
\(G_{n^2}=c\,\chi_{10}^{\,e_n}Q\) are exactly these, with coefficients \(c\,q_{i,j,a,b}\), and
since \(G_{n^2}\) lies in the free algebra \(\C[\psi_4,\psi_6,\chi_{10},\chi_{12}]\), every
exponent \(e_n+j-i\) is nonnegative. In particular
\[
\min\bigl\{u+v: \psi_4^{a}\psi_6^{b}\chi_{10}^{u}\chi_{12}^{v}\in\operatorname{supp}(G_{n^2})\bigr\} = \min_{(i,j)}\bigl((e_n+j-i)+i\bigr)=e_n+\min_j j=e_n,
\]
since some monomial of \(Q\) has \(j=0\).

Recall that \(q_2= e^{2\pi i \tau_2}\).
It remains to show that \(\ord_{q_2}\) of a sum of distinct monomials in the generators equals the minimum of \(u+v\) over the support. By \cref{sec-2}, the Fourier-Jacobi expansions of the generators begin
\[
\begin{split}
\psi_4&=E_4(\tau_1)+O(q_2),\qquad \psi_6=E_6(\tau_1)+O(q_2),\\
\chi_{10}&=c_{10}\,\phi_{10,1}\,q_2+O(q_2^2),\qquad \chi_{12}=c_{12}\,\phi_{12,1}\,q_2+O(q_2^2),
\end{split}
\]
with \(c_{10},c_{12}\in\C^\times\). Orders add under multiplication because the leading coefficients are nonzero holomorphic functions on a domain, so the monomial \(\psi_4^{a}\psi_6^{b}\chi_{10}^{u}\chi_{12}^{v}\) has \(q_2\)-order exactly \(u+v\), with leading coefficient a nonzero multiple of \(E_4^aE_6^b\phi_{10,1}^{u}\phi_{12,1}^{v}\). Let \(e\) be the minimum of \(u+v\) over the support of \(G_{n^2}\). The coefficient of \(q_2^{\,e}\) in \(G_{n^2}\) is a nonzero polynomial expression in \(E_4,E_6,\phi_{10,1},\phi_{12,1}\), which does not vanish by \cref{lem:algindep}. Hence \(\ord_{q_2}(G_{n^2})=e=e_n\).
\end{proof}


\begin{prop}
\label{prop:torsion-vanishing}
Let \(n\geq 2\) and  \(a,e\in\Z\) with \(\gcd(a,e,n)=1\). Then \(G_{n^2}\) vanishes identically on the divisor \(\{nz=a\tau_1+e\}\subset\bH_2\). Consequently every Fourier-Jacobi coefficient \(\phi_m\) of \(G_{n^2}\) satisfies
\[
\phi_m\Bigl(\tau_1,\frac{a\tau_1+e}{n}\Bigr)=0\qquad\text{for all }\tau_1\in\bH.
\]
\end{prop}

\begin{proof}
The quintuple \((a,-n,0,0,e)\) is primitive, since \(\gcd(a,n,e)=1\), and has discriminant \(n^2\). Its singular relation \cref{eq:singular-relation} is \(a\tau_1-nz+e=0\), so the divisor \(\{nz=a\tau_1+e\}\) lies in the preimage of \(H_{n^2}\) in \(\bH_2\), on which the pullback of \(G_{n^2}\) vanishes. Fix \(\tau_1\in\bH\) and \(z=(a\tau_1+e)/n\). The matrix \(\tau\) lies in \(\bH_2\) for every \(\tau_2\) with \(\operatorname{Im}\tau_2\) sufficiently large, and on this open set the convergent series \(\sum_m\phi_m(\tau_1,z)\,q_2^m\) vanishes identically in \(\tau_2\). Uniqueness of Fourier coefficients in \(\tau_2\) gives \(\phi_m(\tau_1,z)=0\) for every \(m\).
\end{proof}

The points \((a\tau_1+e)/n\) appearing in \cref{prop:torsion-vanishing} are exactly the points of exact order \(n\) on the elliptic curve \(\C/\Lambda_{\tau_1}\), and the pairs \((a,e)\) modulo inversion index the fiber of \(X_1(n)\to X(1)\) over the point of \(X(1)\) determined by \(\tau_1\). The proposition thus attaches to each point of \(X_1(n)\) a branch of \(H_{n^2}\) along \(H_1\) on which every Fourier-Jacobi coefficient vanishes; the covering degree \(\nu(n)\) of \cref{subsec:modular-curves} then bounds the vanishing from below.

\begin{prop}
\label{prop:lower-bound}
For every \(n\geq 2\), \(e_n\geq\nu(n)\). Consequently
\[
\degw F_n \;\leq\; k(H_{n^2})\;-\;10\,\nu(n).
\]
\end{prop}

\begin{proof}
Let \(m<\nu(n)\) and let \(\phi_m\) be the \(m\)-th Fourier-Jacobi coefficient of \(G_{n^2}\); it is a Jacobi form of even weight \(k(H_{n^2})\) and index \(m\). Suppose \(\phi_m\neq 0\) and fix \(\tau_1\) with \(\phi_m(\tau_1,\cdot)\not\equiv 0\). The points
\[
x_{a,e}=\frac{a\tau_1+e}{n}\ \bmod\ \Lambda_{\tau_1},\qquad (a,e)\in(\Z/n\Z)^2,\ \gcd(a,e,n)=1,
\]
are pairwise distinct and are exactly the \(\varphi(n)\psi(n)\) points of exact order \(n\) on \(\C/\Lambda_{\tau_1}\). By \cref{prop:torsion-vanishing}, \(\phi_m(\tau_1,\cdot)\) vanishes at every \(x_{a,e}\).

Suppose \(n\geq 3\). No \(x_{a,e}\) is 2-torsion, so the zeros number at least \(\varphi(n)\psi(n)=2\nu(n)>2m\), contradicting the zero count \cite[Thm.~1.2]{EichlerZagier}.

Assume \(n=2\), so \(m\leq 2\). The points \(x_{a,e}\) are the three nontrivial 2-torsion points \(1/2\), \(\tau_1/2\), \((\tau_1+1)/2\). By \cref{lem:2torsion-parity} the vanishing order of \(\phi_m(\tau_1,\cdot)\) at each of them is even, hence at least \(2\), so the zeros number at least \(6>2m\), contradicting the zero count \cite[Thm.~1.2]{EichlerZagier}.

Hence \(\phi_m=0\) for every \(m<\nu(n)\), and \(e_n=\ord_{q_2}(G_{n^2})\geq\nu(n)\) by \cref{thm:fj-order}. The degree bound follows from \cref{thm:closed}.
\end{proof}
\section{The intersection with the product locus and the weight identity}
\label{sec-4}

We will consider 
the upper bound \(e_n\leq\nu(n)\)  in \cref{sec-5} by restricting \(G_{n^2}\) to the product locus, and the restriction vanishes along the curve \(H_{n^2}\cap H_1\). This section describes that curve through modular curves and proves the weight identity that converts its degree into \(\nu(n)\).

A point of \(H_1\) is a product \(E_1\times E_2\), with coordinates the absolute invariants \((j_1,j_2)\). A product lies on \(H_{n^2}\) if and only if it satisfies a primitive singular relation \cref{eq:singular-relation} of discriminant \(n^2\). On the product locus such a relation forces a cyclic isogeny \(E_1\to E_2\) of degree \(M=(n^2-x^2)/(4u^2)\), where \((x,u)\) runs over the pairs occurring in the sum below; see \cite[Ch.~IX]{HvdG}. The intersection \(H_{n^2}\cap H_1\) is therefore a union of the modular curves \(T_M\) of \cref{subsec:modular-curves}, and its degree, counted with one branch per pair, is the divisor sum
\[
\widetilde{R}(n)=\sum_{\substack{x\in\Z,\ |x|<n\\ x\equiv n\ (\mathrm{mod}\ 2)}}\ \sum_{\substack{u\geq 1,\ \gcd(u,x)=1\\ u^2\mid (n^2-x^2)/4}}\psi\Bigl(\frac{n^2-x^2}{4u^2}\Bigr),
\]
where \(x\) runs over both signs, \(x=0\) is counted once, and \(\psi(M)=\deg_{j_2}(T_M)\). Thus \(\widetilde{R}(1)=0\), \(\widetilde{R}(2)=1\), \(\widetilde{R}(3)=6\), \(\widetilde{R}(4)=14\), \(\widetilde{R}(5)=38\), \(\widetilde{R}(6)=48\), \(\widetilde{R}(7)=116\). The precise intersection cycle, with multiplicities, is determined in \cref{lem:cycle} and \cref{cor:sharp}.

\begin{lem}
\label{lem:sigma-psi}
Let \(\sigma_1(N)=\sum_{d\mid N}d\).
For every \(N\geq 1\),
\(
\sigma_1(N)=\sum_{u^2\mid N}\psi\Bigl(\frac{N}{u^2}\Bigr).
\)
\end{lem}

\begin{proof}
Both sides are multiplicative in \(N\), so it suffices to take \(N=p^a\) with \(p\) prime. The right-hand side is
\(
\sum_{0\leq 2j\leq a}\psi(p^{a-2j}).
\)
Since \(\psi(p^e)=p^e+p^{e-1}\) for \(e\geq 1\), the term with \(a-2j\geq 1\) contributes the two powers \(p^{a-2j}\) and \(p^{a-2j-1}\). The term with \(a-2j=0\), present only when \(a\) is even, contributes \(\psi(1)=1\). Each power \(p^i\) with \(0\leq i\leq a\) occurs exactly once, so the sum is \(\sigma_1(p^a)\).
\end{proof}

\begin{lem}
\label{lem:phi-psi-square}
For every \(n\geq 1\),
\(
n^2=\sum_{d\mid n}\varphi(d)\,\psi(d).
\)
This is Jordan's totient identity \(\sum_{d\mid n}J_2(d)=n^2\), since
\(J_2(d)=\varphi(d)\psi(d)\).
\end{lem}
\begin{proof}
The right-hand side is multiplicative, and for \(n=p^a\) it equals
\[
1+\sum_{j=1}^{a}\bigl(p^j-p^{j-1}\bigr)\bigl(p^j+p^{j-1}\bigr)
=1+\sum_{j=1}^{a}\bigl(p^{2j}-p^{2j-2}\bigr)=p^{2a}.\qedhere
\]
\end{proof}

Siegel's formula of \cref{sec-2} expresses \(a_{n^2}\) through the sum \(A(n)\) below. That
sum counts all singular relations of discriminant \(n^2\) on the product locus, primitive or
not. The next proposition stratifies them by content, each relation of content \(h\) being
\(h\) times a primitive relation of discriminant \((n/h)^2\).
 For \(n\geq 1\) set
\[
A(n)=\sum_{\substack{x\in\Z,\ |x|<n\\ x\equiv n\ (\mathrm{mod}\ 2)}}\sigma_1\Bigl(\frac{n^2-x^2}{4}\Bigr),
\]
so that Siegel's formula reads \(a_{n^2}=24\,A(n)+12n^2-2\).

\begin{prop}
\label{prop:stratification}
For every \(n\geq 1\),
\(
A(n)=\sum_{m\mid n}\widetilde{R}(m).
\)
\end{prop}

\begin{proof}
By \cref{lem:sigma-psi},
\(
A(n)=\sum_{(x,u)}\psi\Bigl(\frac{n^2-x^2}{4u^2}\Bigr),
\)
the sum over all pairs \((x,u)\) with \(|x|<n\), \(x\equiv n\pmod 2\), \(u\geq 1\), and \(u^2\mid(n^2-x^2)/4\). Given such a pair, set \(h=\gcd(x,u)\), with the convention \(\gcd(0,u)=u\). From \(4h^2\mid 4u^2\mid n^2-x^2\) and \(h^2\mid x^2\) we get \(h^2\mid n^2\), hence \(h\mid n\). Write \(m=n/h\), \(x=hx'\), \(u=hu'\), so \(\gcd(x',u')=1\). Then
\[
\frac{n^2-x^2}{4}=h^2\cdot\frac{m^2-x'^2}{4},\qquad
\frac{n^2-x^2}{4u^2}=\frac{m^2-x'^2}{4u'^2},
\]
and \(u^2\mid(n^2-x^2)/4\) if and only if \(4u'^2\mid m^2-x'^2\). The last divisibility forces \(m\equiv x'\pmod 2\) and \(u'^2\mid(m^2-x'^2)/4\), and \(|x|<n\) is equivalent to \(|x'|<m\); so \((x',u')\) is a pair counted by \(\widetilde{R}(m)\). 

Conversely, given \(h\mid n\) and a pair \((x',u')\) counted by \(\widetilde{R}(n/h)\), the pair \((hx',hu')\) satisfies all the conditions for \(n\): the parity holds because \(n-hx'=h(m-x')\) with \(m-x'\) even. The two assignments are inverse bijections and preserve the summand, so \(A(n)=\sum_{h\mid n}\widetilde{R}(n/h)\).
\end{proof}

\begin{thm}
\label{thm:weight-identity}
For every \(n\geq 3\),
\[
k(H_{n^2})\;=\;12\,\nu(n)\;+\;12\,\widetilde{R}(n),
\]
and \(k(H_4)=12\,\nu(2)+24\,\widetilde{R}(2)=60\). In particular \(12\mid k(H_{n^2})\) for every \(n\geq 2\).
\end{thm}

\begin{proof}
Define \(B(1)=10\), \(B(2)=60\), and \(B(n)=12\,\nu(n)+12\,\widetilde{R}(n)\) for \(n\geq 3\). The recursion of \cref{sec-2} determines \(k(H_{n^2})\) for every \(n\) from the initial value \(k(H_1)=10\) by induction on \(n\), since \(v(n^2)\neq 0\). It therefore suffices to prove that \(B\) satisfies the same recursion, that is,
\[
\sum_{m\mid n}v(m^2)\,B(m)=\tfrac12\,a_{n^2}=12\,A(n)+6n^2-1
\qquad\text{for every } n\geq 1.
\]
Write \(B(1)=12\,\varphi(1)\psi(1)-2+12\,\widetilde{R}(1)\) and \(B(2)=12\,\varphi(2)\psi(2)+24\,\widetilde{R}(2)\). For the first summands, \(v(1)\cdot\bigl(12\varphi(1)\psi(1)-2\bigr)=6\varphi(1)\psi(1)-1\), \(v(4)\cdot 12\varphi(2)\psi(2)=6\varphi(2)\psi(2)\), and \(v(m^2)\cdot 12\nu(m)=6\varphi(m)\psi(m)\) for \(m\geq 3\); summing over \(m\mid n\) and applying \cref{lem:phi-psi-square} gives
\[
\sum_{m\mid n}6\,\varphi(m)\psi(m)-1=6n^2-1.
\]
For the second summands, \(v(1)\cdot 12\widetilde{R}(1)=0\), \(v(4)\cdot 24\widetilde{R}(2)=12\widetilde{R}(2)\), and \(v(m^2)\cdot 12\widetilde{R}(m)=12\widetilde{R}(m)\) for \(m\geq 3\); summing over \(m\mid n\) and applying \cref{prop:stratification} gives \(12\,A(n)\). Adding the two contributions yields \(12A(n)+6n^2-1\), as required. Divisibility by \(12\) for \(n\geq 3\) is immediate from the identity, and \(k(H_4)=60=12\cdot 5\).
\end{proof}

The identity splits the Humbert weight into two geometric contributions along the product
locus. The term \(12\,\widetilde{R}(n)\) is the degree of the intersection curve
\(H_{n^2}\cap H_1\). The term \(12\,\nu(n)\) is the transversal vanishing along that curve,
indexed by the covering \(X_1(n)\to X(1)\). The two halves are matched in \cref{sec-5}.

\section{The upper bound and the main theorem}
\label{sec-5}


Throughout, \(k=k(H_{n^2})\), which is divisible by \(12\) by \cref{thm:weight-identity}.
%
Let
\begin{equation}
\label{defn:restriction}
g_n(\tau_1,\tau_2):=G_{n^2}\bigl(\begin{smallmatrix}\tau_1&0\\ 0&\tau_2\end{smallmatrix}\bigr),
\end{equation}
the \emph{restriction of the exact Humbert form} to the product locus \(z=0\).
%
The restriction \(g_n\) is a modular form of weight \(k\) in each of \(\tau_1,\tau_2\) separately, symmetric under the swap \(\tau_1\leftrightarrow\tau_2\), since \(\bigl(\begin{smallmatrix}0&I\\ I&0\end{smallmatrix}\bigr)\)-conjugation lies in \(\Sp_4(\Z)\).

\begin{lem}
\label{lem:restriction-bound}
For every $n\geq 2$, 
\(e_n\leq\ord_{q_2}(g_n)\), where \(\ord_{q_2}(g_n)\) is the vanishing order of \(g_n\) in \(q_2\) at the cusp of the second factor.
\end{lem}

\begin{proof}
Write the Fourier-Jacobi expansion \(G_{n^2}=\sum_{m\geq e_n}\phi_m(\tau_1,z)\,q_2^m\), with \(\phi_{e_n}\neq 0\) by \cref{thm:fj-order}. Restricting to \(z=0\) gives \(g_n=\sum_{m\geq e_n}\phi_m(\tau_1,0)\,q_2^m\), so every term of \(g_n\) has \(q_2\)-order at least \(e_n\).
\end{proof}

The inequality can be strict for a single form, since \(\phi_{e_n}(\tau_1,0)\) may vanish identically; the content of the main theorem is that it does not.

\begin{lem}
\label{lem:poly}
Let \(\Delta=\Delta(\tau_1)\), \(\Delta'=\Delta(\tau_2)\), \(j_1=j(\tau_1)\), \(j_2=j(\tau_2)\). Then
\[
f_n(j_1,j_2):=\frac{g_n}{(\Delta\Delta')^{k/12}}
\]
is a nonzero symmetric polynomial in \(\C[j_1,j_2]\), of degree
\(
\deg_{j_2}(f_n)	=\frac{k}{12}-\varsigma_n,
\)
where 
\(
\varsigma_n	:=\ord_{q_2}(g_n).
\)
\end{lem}

\begin{proof}
The form \(g_n\) is not identically zero. Indeed \(H_{n^2}\) and \(H_1\) are distinct
irreducible divisors, so \(H_{n^2}\) does not contain \(H_1\) and \(G_{n^2}\) does not vanish
identically on the product locus.

Fix a generic \(\tau_2\). The function \(\tau_1\mapsto g_n\) is a modular form of weight \(k\)
for \(\SL_2(\Z)\), hence a polynomial in \(E_4,E_6\) of weight \(k\). Dividing by
\(\Delta^{k/12}\) gives a meromorphic modular function of weight zero, holomorphic on \(\bH\),
with a pole of order at most \(k/12\) at the cusp. Such a function is a polynomial in \(j_1\)
of degree at most \(k/12\). The same argument in \(\tau_2\) gives \(f_n\in\C[j_1,j_2]\),
symmetric by the symmetry of \(g_n\).

For the degree, \(\Delta'=q_2+O(q_2^2)\) and \(j_2=q_2^{-1}+744+O(q_2)\). Write
\(f_n=\sum_{d}\gamma_d(j_1)\,j_2^{\,d}\). The \(q_2\)-order of \((\Delta')^{k/12}j_2^{\,d}\) is
\(k/12-d\). Distinct \(d\) therefore contribute distinct \(q_2\)-orders, so
\(\ord_{q_2}(g_n)=k/12-\deg_{j_2}(f_n)\).
\end{proof}

\begin{lem}
\label{lem:cycle}
Let \(n\geq 3\) and let \(f_n\) be as in \cref{lem:poly}. Then
\(
\deg_{j_2}(f_n)\;\geq\;\widetilde{R}(n).
\)
\end{lem}

\begin{proof}
Index the pairs counted by \(\widetilde{R}(n)\) by the unordered pairs \(\{x,-x\}\). A pair
\((0,u)\) is counted once, and a pair with \(x\neq 0\) is counted twice, once for each sign.
Let \((|x|,u)\) be such an index and set \(M=(n^2-x^2)/(4u^2)\geq 1\). Fix a primitive integer
matrix \(P=\bigl(\begin{smallmatrix}p&q\\ r&s\end{smallmatrix}\bigr)\) with \(\det P=M\), and
define the quintuples
\(
w_\pm=\bigl(up,\ \pm x,\ -us,\ ur,\ uq\bigr).
\)
Their discriminants are \(x^2-4(up)(-us)-4(ur)(uq)=x^2+4u^2\det P=n^2\). They are also
primitive, since \(P\) is primitive and the greatest common divisor of the entries is therefore
\(\gcd(u,x)=1\). When \(x=0\) the condition \(\gcd(u,0)=u=1\) forces \(u=1\) and
\(w_+=w_-\). When \(x\neq 0\) the quintuples \(w_+\) and \(w_-\) are distinct and not
proportional.

Let \(D_w\subset\bH_2\) denote the divisor cut out by the singular relation
\cref{eq:singular-relation} attached to a primitive quintuple \(w\) of discriminant \(n^2\).
Each \(D_w\) is smooth. Indeed the gradient of the defining function has entries
\(a-d\tau_2\), \(b+2dz\), \(c-d\tau_1\) in the variables \(\tau_1,z,\tau_2\). Here
\(a-d\tau_2=0\) with \(d\neq 0\) would force \(\tau_2\in\R\), while \(d=0\) leaves
\((a,b,c)\neq(0,0,0)\). By \cref{sec-2} the preimage of \(H_{n^2}\) is the union of the
\(D_w\), so \(G_{n^2}\) vanishes identically on every \(D_w\).

At \(z=0\) the relation attached to \(w_\pm\) reads
\[
up\,\tau_1-us\,\tau_2+ur(-\tau_1\tau_2)+uq=0,
\]
that is,
\( \tau_2=\frac{p\tau_1+q}{r\tau_1+s}=P\tau_1.
\)
Hence \(D_{w_+}\cap\{z=0\}=D_{w_-}\cap\{z=0\}\) is the graph
\(\widetilde{C}=\{(\tau_1,0,P\tau_1)\}\), whose image in \((j_1,j_2)\) is \(T_M\).

Fix a point \(\omega\in\widetilde{C}\) generic in the following sense. The stabilizer of
\(\omega\) in \(\SL_2(\Z)\times\SL_2(\Z)\) is exactly \(\{\pm I\}\times\{\pm I\}\), and
\(\omega\) lies on no divisor \(D_w\) other than those containing \(\widetilde{C}\). Both
conditions exclude countably many proper analytic subsets of \(\widetilde{C}\). Let \(\cO\) be
the local ring of \(\bH_2\) at \(\omega\), a regular local ring and hence a unique
factorization domain. For each primitive quintuple \(w\) of discriminant \(n^2\) with
\(\widetilde{C}\subset D_w\), let \(h_w\in\cO\) be the defining function of \(D_w\). It is a
prime element since \(D_w\) is smooth, and non-associate quintuples give non-associate primes.
Now \(G_{n^2}\) lies in the intersection of the prime ideals \((h_w)\), which in a unique
factorization domain is the principal ideal generated by the product. Hence
\[
G_{n^2}=h_0\cdot\!\!\prod_{w:\ \widetilde{C}\subset D_w}\!\!h_w,\qquad h_0\in\cO.
\]
Restrict to \(z=0\). For every \(w=(a,b,c,d,e)\) in the product,
\(h_w|_{z=0}=a\tau_1+c\tau_2-d\tau_1\tau_2+e\) vanishes to order exactly \(1\) along
\(\widetilde{C}\), since its differential \((a-d\tau_2)\,d\tau_1+(c-d\tau_1)\,d\tau_2\) is
nonzero at \(\omega\). Hence \(g_n\) vanishes along \(\widetilde{C}\) at \(\omega\) to order at
least the number of quintuples \(w\) in the product.

That number is bounded below as follows. The index \((|x|,u)\) contributes \(w_+\) and
\(w_-\), which are distinct when \(x\neq 0\) and coincide when \(x=0\). It therefore
contributes two quintuples in the first case and one in the second, matching the number of
pairs counted by \(\widetilde{R}(n)\) with that index. Distinct indices \((|x|,u)\),
\((|x'|,u')\) with the same value of \(M\) give quintuples that are pairwise non-associate.
Indeed two quintuples built from the same \(P\) are proportional only if their second entries
agree up to sign and their matrix parts have equal content, that is, only if \(|x|=|x'|\) and
\(u=u'\). All these quintuples contain \(\widetilde{C}\) in their divisors and so appear in the
product. The number of quintuples in the product is therefore at least \(\mu_M\), the number of
pairs counted by \(\widetilde{R}(n)\) whose value is \(M\). Finally
\(\Delta(\tau_1)\Delta(\tau_2)\) does not vanish at \(\omega\), and the quotient map to the
\(j\)-coordinates is a local biholomorphism at \(\omega\), so the same lower bound holds for
\(\ord_{T_M}(f_n)\).

Summing over the distinct values \(M\) arising from pairs counted by \(\widetilde{R}(n)\), and
using that distinct \(M\) give distinct irreducible curves \(T_M\) of degree \(\psi(M)\) in
\(j_2\),
\[
\begin{split}
\deg_{j_2}(f_n)\;\geq\;\sum_M \ord_{T_M}(f_n)\cdot\psi(M)
&\;\geq\;\sum_M \mu_M\,\psi(M)\\
&\;=\;\sum_{(x,u)}\psi\Bigl(\frac{n^2-x^2}{4u^2}\Bigr)\;=\;\widetilde{R}(n).\qedhere
\end{split}
\]
\end{proof}

\begin{lem}
\label{lem:n2}
\(\deg_{j_2}(f_2)\;\geq\;2\,\widetilde{R}(2)=2\).
\end{lem}

\begin{proof}
The single pair counted by \(\widetilde{R}(2)\) is \((x,u)=(0,1)\), with \(M=1\) and \(T_1=\{j_1=j_2\}\). The quintuple \(w=(1,0,-1,0,0)\) is primitive of discriminant \(4\), and its singular relation \cref{eq:singular-relation} reads \(\tau_1=\tau_2\). By \cref{sec-2} the divisor \(D_w\) lies in the preimage of \(H_4\), on which \(G_4\) vanishes; restricting to \(z=0\) shows that \(g_2\) vanishes on the diagonal \(\{\tau_1=\tau_2\}\subset\bH\times\bH\). The diagonal maps onto \(T_1\) under \((\tau_1,\tau_2)\mapsto(j_1,j_2)\), so \(f_2\) vanishes on the irreducible curve \(T_1\) and \(j_1-j_2\) divides \(f_2\) in \(\C[j_1,j_2]\). Write \(f_2=(j_1-j_2)h\). By \cref{lem:poly} the polynomial \(f_2\) is symmetric, and \(j_1-j_2\) is antisymmetric, so \(h\) is antisymmetric; hence \(h\) vanishes on \(\{j_1=j_2\}\) and \(j_1-j_2\) divides \(h\). Thus \((j_1-j_2)^2\) divides \(f_2\) and \(\deg_{j_2}(f_2)\geq 2\).
\end{proof}

\begin{thm}[Main Theorem]
\label{thm:main}
For every \(n\geq 2\), \(e_n=\nu(n)\), and there is \(c\in\C^\times\) with
\[
G_{n^2}\;=\;c\;\chi_{10}^{\,\nu(n)}\,F_n(\tau).
\]
Equivalently, \(\degw F_2=30\), and for \(n\geq 3\)
\begin{equation}
\label{eq:degree}
\degw F_n \;=\; k(H_{n^2})\;-\;5\,\varphi(n)\,\psi(n)\;=\;\varphi(n)\,\psi(n)\;+\;12\,\widetilde{R}(n).
\end{equation}
\end{thm}

\begin{proof}
Let \(n=2\). By \cref{lem:restriction-bound}, \cref{lem:poly}, \cref{lem:n2}, and \cref{thm:weight-identity},
\[
e_2\;\leq\;\varsigma_2\;=\;\frac{k(H_4)}{12}-\deg_{j_2}(f_2)\;\leq\;\frac{k(H_4)}{12}-2\,\widetilde{R}(2)\;=\;\nu(2).
\]
Let \(n\geq 3\). By \cref{lem:restriction-bound}, \cref{lem:poly}, \cref{lem:cycle}, and \cref{thm:weight-identity},
\[
e_n\;\leq\;\varsigma_n\;=\;\frac{k(H_{n^2})}{12}-\deg_{j_2}(f_n)\;\leq\;\frac{k(H_{n^2})}{12}-\widetilde{R}(n)\;=\;\nu(n).
\]
By \cref{prop:lower-bound}, \(e_n\geq\nu(n)\) for every \(n\geq 2\). The factorization and the first equality in \cref{eq:degree} follow from \cref{cor:exact-form} and \cref{thm:closed}; the second equality is \cref{thm:weight-identity}.
\end{proof}

\begin{cor}
\label{cor:elementary}
For every \(n\geq 3\),
\[
\degw F_n \;=\; n^2\prod_{p\mid n}\bigl(1-p^{-2}\bigr)\;+\;12\sum_{\substack{x\in\Z,\ |x|<n\\ x\equiv n\ (\mathrm{mod}\ 2)}}\ \sum_{\substack{u\geq 1,\ \gcd(u,x)=1\\ u^2\mid (n^2-x^2)/4}}\psi\Bigl(\frac{n^2-x^2}{4u^2}\Bigr).
\]
\end{cor}

\begin{proof}
The first term is \(\varphi(n)\psi(n)=[\SL_2(\Z):\Gamma_1(n)]\) as in \cref{subsec:modular-curves}, and the double sum is the definition of \(\widetilde{R}(n)\) in \cref{sec-4}; the identity is \cref{eq:degree}.
\end{proof}

\begin{cor}
\label{cor:j10-degree}
For every \(n\geq 2\) the weighted degree \(\degw F_n\) is divisible by \(10\), the coefficient of \(J_{10}^{\degw F_n/10}\) in \(F_n\) is a nonzero constant, and
\[
\deg_{J_{10}} F_n \;=\; \frac{\degw F_n}{10} \;=\; \frac{k(H_{n^2})}{10}\;-\;\nu(n).
\]
In particular \(10\mid k(H_{n^2})\) for every \(n\), and \(\deg_{J_{10}}F_n = 3, 8, 18, 48\) for \(n=2,3,4,5\).
\end{cor}

\begin{proof}
Write \(F_n=\sum_{e\geq 0} c_e(J_2,J_4,J_6)\,J_{10}^{\,e}\). Weighted homogeneity gives \(10e\leq\degw F_n\) for every \(e\) with \(c_e\neq 0\); hence \(\deg_{J_{10}}F_n\leq\degw F_n/10\), with equality if and only if \(F_n(0,0,0,1)\neq 0\), and in that case the extremal coefficient has weight zero, so it is a constant, and \(10\mid\degw F_n\).

The point \([0:0:0:1]\) is the unique point of \(\bP(2,4,6,10)\) with \(J_2=J_4=J_6=0\); it is the moduli point of the curve \(C_0\colon y^2=x^5-1\) \cite{Igusa60}. The automorphism \((x,y)\mapsto(\zeta_5 x,y)\) gives an embedding \(\Q(\zeta_5)\hookrightarrow\End^0(\Jac(C_0))\). If \(\Jac(C_0)\) were isogenous to a product \(E_1\times E_2\) of elliptic curves, then \(\Q(\zeta_5)\) would embed into \(\End^0(E_1\times E_2)\), which is one of \(\Q\times\Q\), \(\Q\times K\), \(K_1\times K_2\), \(M_2(\Q)\), \(M_2(K)\) with \(K,K_1,K_2\) imaginary quadratic; a field mapping to a product embeds into one factor, and every maximal subfield of \(M_2(\Q)\) or \(M_2(K)\) is quadratic over the center, hence contains the center. Since \(\Q(\zeta_5)\) is quartic and its only quadratic subfield \(\Q(\sqrt{5})\) is real, no such embedding exists. Hence \(\Jac(C_0)\) is simple and \([C_0]\notin\cL_n\) for every \(n\geq 2\).

The locus \(\cL_n\) is closed in \(\cM_2\), being the preimage of the closed surface \(H_{n^2}\subset\cA_2\) under the Torelli open immersion. Since \(\cM_2=\{J_{10}\neq 0\}\) is open in \(\bP(2,4,6,10)\) and contains \([0:0:0:1]\), the closure \(\overline{\cL}_n=V(F_n)\) satisfies \(\overline{\cL}_n\cap\cM_2=\cL_n\). Hence \(F_n(0,0,0,1)\neq 0\), and equality holds above. The identity \(\degw F_n=k(H_{n^2})-10\,\nu(n)\) of \cref{thm:main} gives the stated value and the divisibility \(10\mid k(H_{n^2})\).
\end{proof}

\begin{cor}
\label{cor:sharp}
Let \(n\geq 3\) and write \(\phi_m\) for the Fourier-Jacobi coefficients of \(G_{n^2}\). Then:

\begin{enumerate}
\item \(\phi_m=0\) for \(m<\nu(n)\), \(\phi_{\nu(n)}\neq 0\), and for generic \(\tau_1\) the divisor of \(z\mapsto\phi_{\nu(n)}(\tau_1,z)\) on \(\C/\Lambda_{\tau_1}\) is the reduced sum of the \(\varphi(n)\psi(n)\) points of exact order \(n\).

\item \(\varsigma_n=\nu(n)\), and the divisor of \(f_n\) is \(\sum_M\mu_M\,T_M\), where \(M\) runs over the values \((n^2-x^2)/(4u^2)\) of the pairs counted by \(\widetilde{R}(n)\) and \(\mu_M\) is the number of such pairs with that value: each curve \(T_M\) appears with multiplicity exactly \(2\) per unordered pair \(\{x,-x\}\) with \(x\neq 0\), and with multiplicity exactly \(1\) for \(x=0\).
\end{enumerate}
\end{cor}

\begin{proof}
For (1), the equality \(e_n=\nu(n)\) in \cref{thm:main}, together with
\cref{thm:fj-order}, gives \(\phi_m=0\) for \(m<\nu(n)\) and
\(\phi_{\nu(n)}\neq 0\). For generic \(\tau_1\), the function
\(\phi_{\nu(n)}(\tau_1,\cdot)\) is therefore not identically zero. By
\cref{prop:torsion-vanishing}, it vanishes at every point of exact
order \(n\) on \(\C/\Lambda_{\tau_1}\). These are
\(\varphi(n)\psi(n)=2\nu(n)\) distinct points, while a nonzero Jacobi
form of index \(\nu(n)\) has exactly \(2\nu(n)\) zeros, counted with
multiplicity \cite[Thm.~1.2]{EichlerZagier}. Hence all these zeros are
simple and there are no others, proving (1).

For (2), by \cref{thm:main} the chain of inequalities in its proof
collapses, so
\(\varsigma_n=\nu(n)\) and
\(\deg_{j_2}(f_n)=\widetilde{R}(n)\).
Moreover, the argument proving
\(\deg_{j_2}(f_n)\geq\widetilde{R}(n)\) gives
\[
\deg_{j_2}(f_n)
\;\geq\;
\sum_M\ord_{T_M}(f_n)\,\psi(M)
\;\geq\;
\sum_M\mu_M\,\psi(M)
\;=\;
\widetilde{R}(n).
\]
Since the two endpoints are equal, both inequalities are equalities.
The second inequality is a termwise comparison, because
\(\ord_{T_M}(f_n)\geq\mu_M\) for every \(M\). Since
\(\psi(M)>0\), equality of the weighted sums implies
\(\ord_{T_M}(f_n)=\mu_M\) for every \(M\).

Set \(\Pi:=\prod_M\Phi_M^{\mu_M}\). Then \(\Pi\) divides \(f_n\), and
the quotient \(h:=f_n/\Pi\) has degree zero in \(j_2\), hence lies in
\(\C[j_1]\). For \(n\geq 3\), the multiplicity \(\mu_1\) is even.
Indeed, a pair counted by \(\widetilde{R}(n)\) with \(x=0\) and
\(M=1\) would force \(n=2u\) and
\(u=\gcd(0,u)=1\), hence \(n=2\). Thus every occurrence of \(M=1\)
comes from a pair with \(x\neq 0\) and its mirror. Each \(\Phi_M\)
with \(M\geq 2\) satisfies
\(\Phi_M(j_2,j_1)=\Phi_M(j_1,j_2)\), while
\(\Phi_1(j_2,j_1)=-\Phi_1(j_1,j_2)\). Consequently
\(\Pi(j_2,j_1)=(-1)^{\mu_1}\Pi(j_1,j_2)=\Pi(j_1,j_2)\).
Combined with the symmetry
\(f_n(j_2,j_1)=f_n(j_1,j_2)\) from \cref{lem:poly}, this gives
\[
h(j_1)
\;=\;
\frac{f_n(j_1,j_2)}{\Pi(j_1,j_2)}
\;=\;
\frac{f_n(j_2,j_1)}{\Pi(j_2,j_1)}
\;=\;
h(j_2)
\]
in \(\C(j_1,j_2)\), so \(h\) is a nonzero constant. Therefore
\(\Div(f_n)=\sum_M\mu_M\,T_M\), proving (2).
\end{proof}

\begin{cor}
\label{cor:f2}
\(f_2=c\,(j_1-j_2)^2\) for some \(c\in\C^\times\), that is,
\[
g_2(\tau_1,\tau_2)\;=\;c\,\bigl(\Delta(\tau_1)\Delta(\tau_2)\bigr)^{5}\bigl(j(\tau_1)-j(\tau_2)\bigr)^{2}.
\]
\end{cor}

\begin{proof}
By \cref{thm:main} the chain of inequalities for \(n=2\) collapses: \(\varsigma_2=\nu(2)=3\) and \(\deg_{j_2}(f_2)=2\). By \cref{lem:n2}, \((j_1-j_2)^2\) divides \(f_2\), and the quotient has degree zero in each variable, hence is a nonzero constant. The display restates \(f_2=c(j_1-j_2)^2\) via \cref{lem:poly} with \(k(H_4)/12=5\).
\end{proof}

The corollary matches the two halves of the weight identity: the divisor of \(f_n\), of total degree \(\widetilde{R}(n)\) in \(j_2\), is the intersection cycle \(H_{n^2}\cdot H_1\) in the coordinates \((j_1,j_2)\), and the divisor of the leading Fourier-Jacobi coefficient, of degree \(2\nu(n)\) on each elliptic curve, is the fiber of \(X_1(n)\to X(1)\) counted through both signs.

\section{Boundary data and reconstruction}
\label{sec-6}


Throughout, \(n\geq 2\), \(k=k(H_{n^2})\), and \(g_n\), \(f_n\), \(\mu_M\) are as in \cref{sec-5}, and \(T_M\), \(\Phi_M\) as in \cref{subsec:modular-curves}.

\begin{lem}
\label{cor:boundary-product}
Let \(n\geq 3\). Up to a nonzero scalar,
\begin{equation}
\label{eq:P}
g_n(\tau_1,\tau_2)\;=\;\bigl(\Delta(\tau_1)\,\Delta(\tau_2)\bigr)^{k/12}\,\prod_{M}\Phi_M\bigl(j(\tau_1),j(\tau_2)\bigr)^{\mu_M},
\end{equation}
the product over the values \(M=(n^2-x^2)/(4u^2)\) of the pairs counted by \(\widetilde{R}(n)\), with the multiplicities \(\mu_M\) of \cref{cor:sharp}. For \(n=2\) the same statement holds with \(\Phi_1=j_1-j_2\) and \(\mu_1=2\), which is \cref{cor:f2}.
\end{lem}

\begin{proof}
Set \(\Pi=\prod_M\Phi_M^{\mu_M}\). By \cref{cor:sharp}, \(\Div(f_n)=\sum_M\mu_M\,T_M\). Each \(\Phi_M\) is irreducible, so \(\Phi_M^{\mu_M}\) divides \(f_n\), and the \(\Phi_M\) are pairwise coprime, so \(\Pi\) divides \(f_n\). By  \cref{lem:poly}, 
\[
\deg_{j_2}(f_n)=\widetilde{R}(n)=\sum_M\mu_M\,\psi(M)=\deg_{j_2}(\Pi),
\]
 and by the symmetry of \(f_n\) the same holds in \(j_1\). The quotient \(f_n/\Pi\) is a polynomial of degree zero in each variable, hence a nonzero constant, and \cref{eq:P} restates \(f_n=c\,\Pi\) via \cref{lem:poly}.
\end{proof}

The irreducible components of the intersection cycle \(H_{n^2}\cdot H_1\) are the modular curves \(T_M\) of cyclic \(M\)-isogenies, each appearing with multiplicity \(\mu_M\) as determined by the pairs counted in \(\widetilde{R}(n)\) (\cref{cor:sharp} and \cref{cor:boundary-product}). Each such curve is birational to the classical modular curve \(X_0(M)\), which is its normalization, and is singular for \(M\geq 2\). Its geometric genus is therefore the genus of \(X_0(M)\), and \(X_0(M)\) has genus zero if and only if
\(
M\in\{1,2,3,4,5,6,7,8,9,10,12,13,16,18,25\}.
\)
Consequently, for every \(n\geq 2\) the genus-zero irreducible components of the boundary divisor \(H_{n^2}\cap H_1\) are precisely the curves \(T_M\) for those \(M\) that appear among the values \((n^2-x^2)/(4u^2)\) of the pairs counted by \(\widetilde{R}(n)\) and that lie in the list above; their multiplicities are the corresponding \(\mu_M\).

The right-hand side of \cref{eq:P} has weight \(k\) in each of \(\tau_1,\tau_2\), since \(12\mid k\) by \cref{thm:weight-identity}, and its \(q_2\)-order is
\[
k/12-\sum_M\mu_M\,\psi(M)=k/12-\widetilde{R}(n)=\nu(n),
\]
since each \(\Phi_M\) is monic of degree \(\psi(M)\) in \(j_2\); this recovers \(\varsigma_n=\nu(n)\)  from \cref{eq:P}.

The leading Fourier-Jacobi coefficient admits a closed formula as well. Let
\[
\theta_{11}(\tau_1,z)=\sum_{i\in\Z}(-1)^i\,q_1^{(2i+1)^2/8}\,\zeta^{(2i+1)/2}
\]
be the classical odd theta function; it vanishes exactly on \(\Lambda_{\tau_1}=\Z\tau_1+\Z\), to first order, and satisfies
\[
\theta_{11}(\tau_1,z+\ell\tau_1+\ell')=(-1)^{\ell+\ell'}\,e^{-\pi i(\ell^2\tau_1+2\ell z)}\,\theta_{11}(\tau_1,z),\qquad \ell,\ell'\in\Z.
\]

\begin{prop}
\label{prop:theta}
Let \(n\geq 3\) and let \(\tau_1\) be generic in the sense of \cref{cor:sharp}. Let \(E_n\subset\C/\Lambda_{\tau_1}\) be the set of the \(\varphi(n)\psi(n)\) points of exact order \(n\), a disjoint union of \(\nu(n)\) inversion orbits \(\{x,-x\}\). Then
\[
\phi_{\nu(n)}(\tau_1,z)\;=\;\phi_{\nu(n)}(\tau_1,0)\,\prod_{\{x,-x\}\subset E_n}\frac{\theta_{11}(\tau_1,z-x)\,\theta_{11}(\tau_1,z+x)}{-\,\theta_{11}(\tau_1,x)^2},
\]
and \(\phi_{\nu(n)}(\tau_1,0)=c\,\Delta(\tau_1)^{k/12}\) with \(c\in\C^\times\) independent of \(\tau_1\).
\end{prop}

\begin{proof}
For each orbit the transformation law gives
\[
\begin{split}
&\theta_{11}(\tau_1,z-x+\ell\tau_1+\ell')\,\theta_{11}(\tau_1,z+x+\ell\tau_1+\ell')\\
&\qquad=e^{-2\pi i(\ell^2\tau_1+2\ell z)}\,\theta_{11}(\tau_1,z-x)\,\theta_{11}(\tau_1,z+x),
\end{split}
\]
the \(x\)-dependence of the two exponential factors cancelling; so the product over the \(\nu(n)\) orbits transforms under \(z\mapsto z+\ell\tau_1+\ell'\) with the factor \(e^{-2\pi i\nu(n)(\ell^2\tau_1+2\ell z)}\), the elliptic law of index \(\nu(n)\) satisfied by \(\phi_{\nu(n)}(\tau_1,\cdot)\). The zero divisor of the product on \(\C/\Lambda_{\tau_1}\) is the reduced sum of the points of \(E_n\), which equals the divisor of \(\phi_{\nu(n)}(\tau_1,\cdot)\) by \cref{cor:sharp}. The ratio of the two sides is therefore elliptic, holomorphic, and nowhere zero, hence constant in \(z\), and equals \(1\) at \(z=0\), where the product is normalized by \(\theta_{11}(\tau_1,-x)=-\theta_{11}(\tau_1,x)\neq 0\) since no \(x\in E_n\) is 2-torsion for \(n\geq 3\). Since
\[
g_n(\tau_1,\tau_2)=\sum_m\phi_m(\tau_1,0)\,q_2^m,
\]
the value \(\phi_{\nu(n)}(\tau_1,0)\) is the coefficient of \(q_2^{\nu(n)}\) in \cref{eq:P}. Each \(\Phi_M\) is monic of degree \(\psi(M)\) in \(j_2\), except \(\Phi_1=j_1-j_2\), whose leading coefficient in \(j_2\) is \(-1\); so \(\prod_M\Phi_M^{\mu_M}\) has leading coefficient \(\pm 1\) in \(j_2\) and degree \(\widetilde{R}(n)\). With \(\Delta(\tau_2)^{k/12}=q_2^{k/12}\bigl(1+O(q_2)\bigr)\) and \(j(\tau_2)=q_2^{-1}\bigl(1+O(q_2)\bigr)\),
\[
\Delta(\tau_2)^{k/12}\prod_M\Phi_M\bigl(j_1,j(\tau_2)\bigr)^{\mu_M}  =\pm\,q_2^{\,k/12-\widetilde{R}(n)}\bigl(1+O(q_2)\bigr)  =\pm\,q_2^{\,\nu(n)}\bigl(1+O(q_2)\bigr),
\]
the leading coefficient being independent of \(j_1\). Hence the coefficient of \(q_2^{\nu(n)}\) in \cref{eq:P} is a nonzero constant multiple of \(\Delta(\tau_1)^{k/12}\).
\end{proof}

The boundary data of \(G_{n^2}\) are thus determined completely: the weight \(k\) by \cref{thm:weight-identity}, the restriction \(g_n\) by \cref{cor:boundary-product}, the Fourier-Jacobi order \(e_n=\nu(n)\) by \cref{thm:main}, and the leading coefficient \(\phi_{\nu(n)}\) by \cref{prop:theta}. The reconstruction problem is to recover \(G_{n^2}\), hence \(F_n=c\,G_{n^2}/\chi_{10}^{\nu(n)}\) by \cref{thm:main}, from these data.


\begin{rem}\label{rem:ambiguity}
The boundary data alone do not determine \(G_{n^2}\) inside \(M_k\): for every \(W\in M_{k-10(\nu(n)+1)}\) the form \(G_{n^2}+\chi_{10}^{\nu(n)+1}W\) has the same four data, since \(\chi_{10}^{\nu(n)+1}W\) has weight \(k\), lies in \(\ker\rho\), and has \(q_2\)-order at least \(\nu(n)+1\).
In Igusa coordinates the ambiguity is \(F_n\mapsto F_n+J_{10}W\), the tautological one: the boundary data prescribe the behaviour of \(F_n\) along \(H_1\) and no more. Any characterization of \(G_{n^2}\) must therefore include an interior condition.
\end{rem}

\begin{lem}\label{lem:principal}
Let \(F\) be a Siegel modular form of even weight \(k\) and level one with \(F\bigl(\begin{smallmatrix}\tau_1&0\\0&\tau_2\end{smallmatrix}\bigr)\equiv 0\). Then \(F=\chi_{10}F'\) with \(F'\) a Siegel modular form of weight \(k-10\). Equivalently, the kernel of the restriction homomorphism \(\rho\colon M^{\mathrm{even}}\to\cO(\bH\times\bH)\), \(\rho(F)(\tau_1,\tau_2)=F\bigl(\begin{smallmatrix}\tau_1&0\\0&\tau_2\end{smallmatrix}\bigr)\), is the principal ideal \((\chi_{10})\).
\end{lem}

\begin{proof}
The map \(\rho\) is a ring homomorphism, and \(\rho(\chi_{10})=0\) since \(\Div(\chi_{10})=H_1\) and \(H_1\) is the image of \(\{z=0\}\) (\cref{sec-2}). We compute \(\rho\) on the remaining generators of \cref{sec-2}.

For fixed \(\tau_2\), the function \(\tau_1\mapsto\rho(\psi_4)\) is a holomorphic modular form of weight \(4\) and level one, so \(\rho(\psi_4)=c(\tau_2)E_4(\tau_1)\), and \(c\) is a modular form of weight \(4\), so \(\rho(\psi_4)=c\,E_4(\tau_1)E_4(\tau_2)\) with \(c\in\C\); letting \(\tau_2\to i\infty\) and using that the Siegel operator sends \(\psi_4\) to \(E_4\) (\cref{sec-2}) gives \(c=1\). The same argument gives \(\rho(\psi_6)=E_6(\tau_1)E_6(\tau_2)\).

For \(\chi_{12}\), the same argument gives
\(
\rho(\chi_{12})=A_1(\tau_2)E_4(\tau_1)^3+A_2(\tau_2)\Delta(\tau_1)
\)
with \(A_1,A_2\in\cE_{12}\). Letting \(\tau_2\to i\infty\) and using that the Siegel operator annihilates \(\chi_{12}\) (\cref{sec-2}) gives
\(
A_1(i\infty)E_4^3+A_2(i\infty)\Delta=0,
\)
hence \(A_1(i\infty)=A_2(i\infty)=0\), hence \(A_1=a\Delta(\tau_2)\) and \(A_2=b\Delta(\tau_2)\) with \(a,b\in\C\); the symmetry of \(\rho(\chi_{12})\) in \((\tau_1,\tau_2)\) forces \(a=0\), so \(\rho(\chi_{12})=b\,\Delta(\tau_1)\Delta(\tau_2)\). The coefficient of \(q_2\) in \(\rho(\chi_{12})\) is \(c_{12}\phi_{12,1}(\tau_1,0)\) with \(c_{12}\neq 0\) (\cref{sec-2}), and \(\phi_{12,1}\) is a nonzero constant multiple of \(\Delta\wphi_{0,1}\) \cite[Thm.~9.3]{EichlerZagier}; the restriction \(\wphi_{0,1}(\tau_1,0)\) is a holomorphic modular form of weight \(0\), hence the constant \(12\) in the normalization of \cite{EichlerZagier}. Hence \(b\neq 0\).

The three forms \(E_4E_4'\), \(E_6E_6'\), \(\Delta\Delta'\) are algebraically independent, where the prime denotes evaluation at \(\tau_2\). Indeed, in a polynomial relation
\[
\sum_{j\geq 0}P_j(E_4E_4',E_6E_6')\,(\Delta\Delta')^j=0,
\]
the coefficient of \(q_1^0\) is \(P_0\bigl(E_4(\tau_2),E_6(\tau_2)\bigr)=0\), so \(P_0=0\) by the algebraic independence of \(E_4\) and \(E_6\); dividing by \(\Delta\Delta'\neq 0\) and iterating kills every \(P_j\).

Now write \(F=P(\psi_4,\psi_6,\chi_{10},\chi_{12})\), using \(M^{\mathrm{even}}=\C[\psi_4,\psi_6,\chi_{10},\chi_{12}]\) (\cref{sec-2}). The hypothesis reads \(P(E_4E_4',E_6E_6',0,\,b\,\Delta\Delta')=0\), so \(P(x,y,0,w)=0\) as a polynomial by the independence just proved. In the polynomial ring this is exactly the divisibility \(\chi_{10}\mid P\), and \(F'=F/\chi_{10}\) has weight \(k-10\).
\end{proof}

\begin{rem}\label{rem:odd}
The parity hypothesis in \cref{lem:principal} is necessary: every Siegel modular form of odd weight vanishes identically on \(\{z=0\}\), as in \cref{sec-2}; in particular \(\chi_{35}\) does. But \(\chi_{10}\nmid\chi_{35}\): since the full ring is generated over \(M^{\mathrm{even}}\) by \(\chi_{35}\) (\cref{sec-2}) and \(\chi_{35}^2\) is even, every form of odd weight is \(\chi_{35}\) times a form of even weight, so the smallest odd weight is \(35\), while \(\chi_{35}/\chi_{10}\) would have odd weight \(25\). Hence in the full ring the ideal of forms vanishing on \(H_1\) is \((\chi_{10},\chi_{35})\), and it is principal exactly on the even part. Since \(12\mid k(H_{n^2})\) by \cref{thm:weight-identity}, the even case is the one relevant to this paper.
\end{rem}

\begin{cor}\label{cor:taylor}
Set \(\cR=\C[\psi_4,\psi_6,\chi_{12}]\). Every Siegel modular form \(F\) of even weight \(k\) and level one admits a unique expansion
\[
F\;=\;\sum_{e=0}^{\lfloor k/10\rfloor}\chi_{10}^{\,e}\,F_e,\qquad F_e\in \cR\ \text{of weight }k-10e,
\]
and the restriction \(\rho\) is injective on \(\cR\). Hence \(F\) is determined by the symmetric bi-modular forms \(\rho(F_0),\rho(F_1),\dots\), and \(F\) vanishes to order at least \(e_0\) along \(H_1\) if and only if \(F_e=0\) for \(e<e_0\).
\end{cor}

\begin{proof}
The expansion exists and is unique because \(M^{\mathrm{even}}=\cR[\chi_{10}]\) is a polynomial ring over \(\cR\) (\cref{sec-2}). By \cref{lem:principal}, \(\ker\rho=(\chi_{10})\), and \(\cR\cap(\chi_{10})=0\) in a polynomial ring, so \(\rho\) is injective on \(\cR\). The last statement holds because \(\rho(F_e)\) is the coefficient of the order-\(e\) term of \(F\) along \(H_1\) in the parameter \(\chi_{10}\).
\end{proof}

\begin{lem}\label{lem:descent}
Let \(n\geq 2\), let \(F\) be a Siegel modular form of even weight with \(\chi_{10}\mid F\), and set \(F'=F/\chi_{10}\). Let \(L\geq 1\) and suppose that for every \(m\leq L\) the Fourier-Jacobi coefficient \(\phi_m(F)\) vanishes at every point of exact order \(n\). Then for every \(m\leq L-1\) the coefficient \(\phi_m(F')\) vanishes at every point of exact order \(n\).
\end{lem}

\begin{proof}
Since \(\ord_{q_2}(\chi_{10})=1\), the product expansion gives
\[
\phi_m(F)=\sum_{i=1}^{m}\phi_i(\chi_{10})\,\phi_{m-i}(F')
\]
for every \(m\geq 1\), an identity of functions of \((\tau_1,z)\). The leading coefficient is \(\phi_1(\chi_{10})=c_{10}\phi_{10,1}\) with \(c_{10}\neq 0\) (\cref{sec-2}), and \(\phi_{10,1}\) is a nonzero constant multiple of \(\Delta\wphi_{-2,1}\) \cite[Thm.~9.3]{EichlerZagier}, whose divisor in \(z\) is \(2\cdot(0)\) by \cref{sec-2}; hence \(\phi_1(\chi_{10})(\tau_1,x)\neq 0\) for every \(\tau_1\) and every \(x\notin\Lambda_{\tau_1}\).

We induct on \(m\). For \(m=1\), evaluating \(\phi_1(F)=\phi_1(\chi_{10})\phi_0(F')\) at a point \(x\) of exact order \(n\) gives \(\phi_0(F')(\tau_1,x)=0\). For \(2\leq m\leq L\), evaluating at \(x\) and using the inductive hypothesis \(\phi_{m-i}(F')(\tau_1,x)=0\) for \(2\leq i\leq m\) leaves
\[
\phi_1(\chi_{10})(\tau_1,x)\,\phi_{m-1}(F')(\tau_1,x)=0,
\]
hence \(\phi_{m-1}(F')(\tau_1,x)=0\).
\end{proof}

\Cref{lem:principal} and \cref{cor:taylor} replace the boundary-reconstruction heuristic by an unconditional statement: an even Siegel modular form of genus two is determined by its Taylor expansion along \(H_1\), the Taylor coefficients lie in \(\cR\), and each is faithfully recorded by its restriction to \(H_1\). The uniqueness statement with content is \cref{prop:linear}, which characterizes the scalar multiples of \(G_{n^2}\) inside \(M_k\) by the torsion conditions across all Fourier-Jacobi levels; \cref{lem:descent} is the mechanism by which those conditions propagate through the Taylor expansion, consuming one level per division by \(\chi_{10}\).

One interior condition suffices, and it is linear.

\begin{prop}
\label{prop:linear}
Let \(n\geq 2\) and let \(a,e\in\Z\) with \(\gcd(a,e,n)=1\). A form \(G\in M_{k}\) vanishes identically on \(\{nz=a\tau_1+e\}\subset\bH_2\) if and only if \(G\) is a scalar multiple of \(G_{n^2}\).
\end{prop}

\begin{proof}
The quintuple \(w=(a,-n,0,0,e)\) is primitive of discriminant \(n^2\) and cuts out the divisor \(D_w=\{nz=a\tau_1+e\}\), on which \(G_{n^2}\) vanishes by \cref{sec-2}. 
The assertion is immediate for \(G=0\), so suppose that
\(G\in M_k\) is nonzero and vanishes on \(D_w\).
The group \(\Sp_4(\Z)\) acts transitively on the primitive quintuples of discriminant \(n^2\) \cite{Humbert1899}, see also \cite[Ch.~IX]{HvdG}, and modularity of \(G\) transports the  vanishing along \(D_w\) to every divisor in the preimage of  \(H_{n^2}\). 
Hence
\(
\Div(G)\geq H_{n^2}
\).

It follows that \(G/G_{n^2}\) is a regular modular function of weight zero on \(\cA_2\). As the quotient of two modular forms of the
same weight, it is a rational function on the Baily--Borel compactification \(\cA_2^\ast\). The boundary \(\cA_2^\ast\setminus\cA_2\) has codimension two, and \(\cA_2^\ast\) is normal, so \(G/G_{n^2}\) extends to a regular function on \(\cA_2^\ast\). Since \(\cA_2^\ast\) is projective, this function is constant. Therefore \(G=cG_{n^2}\) for some
\(c\in\C^\times\).
\end{proof}

For \(n\geq 2\) and \(k=k(H_{n^2})\), let \(m_0(n,k)\) be the least integer with the following property: a form \(G\in M_k\) whose Fourier-Jacobi coefficients \(\phi_m\) vanish at every point of exact order \(n\) for all \(m\leq m_0(n,k)\) is a scalar multiple of \(G_{n^2}\). Such an integer exists. The conditions are linear and \(M_k\) is finite dimensional, so the solution spaces of the truncated systems stabilize; a form in the stable space has \(\sum_m\phi_m(\tau_1,(a\tau_1+e)/n)\,q_2^m\equiv 0\), hence vanishes on \(\{nz=a\tau_1+e\}\), and \cref{prop:linear} applies.


We write \(J_{k,m}\) for the space of Jacobi forms of weight \(k\) and index \(m\) for \(\SL_2(\Z)\ltimes\Z^2\), and \(q_1=e^{2\pi i\tau_1}\), \(\zeta=e^{2\pi iz}\).

Fix \(n\geq 2\) and \(k=k(H_{n^2})\). For an integer \(m>\nu(n)\) define
\[
\begin{split}
S_m=\bigl\{\omega\in J_{k,m}\;:\;\ &\ord_{q_1}(\omega)\geq m,\\
&\omega(\tau_1,x)=0\ \text{for every }x\in\C/\Lambda_{\tau_1}\text{ of exact order }n,\\
&\omega(\tau_1,0)=0\bigr\}.
\end{split}
\]
Two forms in \(J_{k,m}\) with the same Fourier coefficients \(c(N,r)\) for \(N<m\), the same restriction to \(z=0\), and vanishing at every point of exact order \(n\) differ by an element of \(S_m\). The determination of \(\phi_m\) from the level-\(m\) data alone is therefore unique if and only if \(S_m=0\).

\begin{prop}\label{prop:tail}
Let \(n\geq 2\), let \(k=k(H_{n^2})\), and let \(m\geq k/10\) be an integer. Then \(S_m=0\).
\end{prop}

\begin{proof}
Let \(\omega\in S_m\). The form \(\omega/\Delta^m\) is holomorphic on \(\bH\times\C\), has weight \(k-12m\) and index \(m\), and its Fourier expansion has exponents \(N\geq 0\); hence it is a weak Jacobi form of even weight. By the structure theorem \cite[Thm.~9.3]{EichlerZagier},
\[
\omega/\Delta^{m}\;\in\;\bigoplus_{a=0}^{m}\cE_{k-12m+2a}\,\wphi_{-2,1}^{\,a}\,\wphi_{0,1}^{\,m-a}.
\]
If \(m>k/10\), then \(k-12m+2a<0\) for every \(0\leq a\leq m\) and \(\omega=0\). If \(m=k/10\), only the summand \(a=m\) survives and \(\omega=c\,\Delta^m\wphi_{-2,1}^{\,m}\) with \(c\in\C\). The divisor of \(\wphi_{-2,1}(\tau_1,\cdot)\) is \(2\cdot(0)\) by \cref{sec-2}, and points of exact order \(n\) do not lie in \(\Lambda_{\tau_1}\), so the vanishing condition defining \(S_m\) forces \(c=0\).
\end{proof}

\begin{thm}\label{thm:obstruction}
Let \(n\geq 3\) be prime, let \(k=k(H_{n^2})\), and let \(m=\nu(n)+1=(n^2+1)/2\). Let \(\eta(\tau_1)=q_1^{1/24}\prod_{i\geq 1}(1-q_1^i)\), so that \(\Delta=\eta^{24}\), and set
\[
\Xi_n(\tau_1,z)\;=\;\eta(\tau_1)^{18}\,\theta_{11}(\tau_1,z)\,\theta_{11}(\tau_1,nz).
\]
Then \(\Xi_n\) is a Jacobi cusp form of weight \(10\) and index \(m\) with \(\ord_{q_1}(\Xi_n)=1\), and
\[
S_{m} \;=\;\Xi_n\cdot\Delta^{\nu(n)}\cdot \cE_{12\widetilde{R}(n)-10},
\]
and \(\dim S_{m}\;=\;\widetilde{R}(n)-1\;\geq\;5\).
\end{thm}

\begin{proof}
For prime \(n\) one has \(\varphi(n)\psi(n)=n^2-1\), so \(\nu(n)+1=(n^2+1)/2=m\).

The elliptic transformation of \(\theta_{11}\), applied at \(z\) and at \(nz\), produces the sign \((-1)^{(n+1)(\ell+\ell')}=1\), since \(n\) is odd, and the factor \(e^{-2\pi im(\ell^2\tau_1+2\ell z)}\); so \(\Xi_n\) satisfies the index-\(m\) elliptic law.
From \((2i+1)^2\equiv 1\pmod 8\) one has
\[
\theta_{11}(\tau_1+1,z)=e^{\pi i/4}\theta_{11}(\tau_1,z),
\]
and \(\eta(\tau_1+1)=e^{\pi i/12}\eta(\tau_1)\), so the factor of \(\Xi_n\) under \(\tau_1\mapsto\tau_1+1\) is \(e^{3\pi i/2}e^{\pi i/4}e^{\pi i/4}=1\).

The inversion
\[
\theta_{11}(-1/\tau_1,w/\tau_1)=-i(-i\tau_1)^{1/2}e^{\pi iw^2/\tau_1}\theta_{11}(\tau_1,w),
\]
applied at \(w=z\) and \(w=nz\), together with \(\eta(-1/\tau_1)=(-i\tau_1)^{1/2}\eta(\tau_1)\), gives
\[
\Xi_n(-1/\tau_1,z/\tau_1)=\tau_1^{10}e^{2\pi imz^2/\tau_1}\Xi_n(\tau_1,z).
\]
Since the two transformations generate \(\SL_2(\Z)\), the form \(\Xi_n\) transforms as a Jacobi form of weight \(10\) and index \(m\) with trivial character.

The Fourier exponents of \(\Xi_n\) are \(N=(s^2+t^2)/2+3/4+j\) and \(r=s+nt\) with \(s,t\in\Z+\tfrac12\) and \(j\geq 0\), and
\[
4mN-r^2\;=\;(ns-t)^2+2(n^2+1)\bigl(\tfrac34+j\bigr)\;\geq\;\tfrac32\,(n^2+1)\;>\;0,
\]
so \(\Xi_n\) is a Jacobi cusp form. Its \(q_1\)-order is \(3/4+1/8+1/8=1\), and the coefficient of \(q_1\) is the nonzero Laurent polynomial
\[
\zeta^{(n+1)/2}-\zeta^{(n-1)/2}-\zeta^{-(n-1)/2}+\zeta^{-(n+1)/2}.
\]

For fixed \(\tau_1\) the zeros of \(\theta_{11}(\tau_1,\cdot)\) are the points of \(\Lambda_{\tau_1}\), all simple; hence the divisor of \(z\mapsto\Xi_n(\tau_1,z)\) on \(\C/\Lambda_{\tau_1}\) is \(2\cdot(0)+\sum_x(x)\), the sum running over the \(n^2-1\) nonzero \(n\)-torsion points, each simple. Since \(n\) is prime, every nonzero \(n\)-torsion point has exact order \(n\).

Let \(\omega\in S_m\) with \(\omega\neq 0\), and fix \(\tau_1\) with \(\omega(\tau_1,\cdot)\not\equiv 0\). The weight \(k\) is even, so \(\omega(\tau_1,\cdot)\) is even in \(z\), and the conditions defining \(S_m\) give a zero of order at least \(2\) at \(z=0\) and zeros at the \(n^2-1\) points of exact order \(n\); that is at least \(n^2+1=2m\) zeros, and by \cite[Thm.~1.2]{EichlerZagier} exactly \(2m\).
Hence the divisor of \(\omega(\tau_1,\cdot)\) equals the divisor of \(\Xi_n(\tau_1,\cdot)\). The quotient \(F=\omega/\Xi_n\) is, for each such \(\tau_1\), holomorphic in \(z\) and invariant under \(\Lambda_{\tau_1}\), since both forms satisfy the same index-\(m\) elliptic law; a holomorphic elliptic function is constant, so \(F\) depends on \(\tau_1\) alone. Since \(\Xi_n(\tau_1,\cdot)\not\equiv 0\) for every \(\tau_1\), the function \(F\) is holomorphic on \(\bH\), and it transforms with weight \(k-10\) and trivial character, because \(\omega\) and \(\Xi_n\) share the automorphy factor \(e^{2\pi imcz^2/(c\tau_1+d)}\). Orders add in \(q_1\) because the leading coefficient of \(\Xi_n\) is a nonzero Laurent polynomial, so \(\ord_{q_1}(F)=\ord_{q_1}(\omega)-1\geq m-1=\nu(n)\geq 0\). Hence \(F\in\cE_{k-10}\) and \(\Delta^{\nu(n)}\) divides \(F\). By \cref{thm:weight-identity}, \(k-10-12\nu(n)=12\widetilde{R}(n)-10\), so \(\omega\in\Xi_n\Delta^{\nu(n)}\cE_{12\widetilde{R}(n)-10}\).

Conversely, let \(F\in\cE_{12\widetilde{R}(n)-10}\) and set \(\omega=\Xi_n\Delta^{\nu(n)}F\). Then \(\omega\in J_{k,m}\), and \(\ord_{q_1}(\omega)=1+\nu(n)+\ord_{q_1}(F)\geq m\); moreover \(\theta_{11}(\tau_1,nx)=0\) for every \(n\)-torsion point \(x\), so \(\omega\) vanishes at every point of exact order \(n\) and at \(z=0\). Hence \(\omega\in S_m\). The map \(F\mapsto\Xi_n\Delta^{\nu(n)}F\) is injective, so \(S_m=\Xi_n\Delta^{\nu(n)}\cE_{12\widetilde{R}(n)-10}\) and \(\dim S_m=\dim\cE_{12\widetilde{R}(n)-10}\).

Finally, \(12\widetilde{R}(n)-10\equiv 2\pmod{12}\), so \(\dim\cE_{12\widetilde{R}(n)-10}=\widetilde{R}(n)-1\). The pairs \((x,u)=(\pm(n-2),1)\) are counted by \(\widetilde{R}(n)\), each with summand \(\psi(n-1)\), so \(\widetilde{R}(n)\geq 2\psi(n-1)\geq 6\) and \(\dim S_m\geq 5\).
\end{proof}

\begin{cor}\label{cor:S5}
Let \(n=3\), so that \(k=120\), \(\nu(3)=4\), \(\widetilde{R}(3)=6\), and \(m=5\). Then
\[
S_5\;=\;\Xi_3\cdot\Delta^4\cdot \cE_{62},\qquad \dim S_5=5,
\]
with basis \(\Xi_3\,\Delta^{4+i}\,E_6E_4^{14-3i}\) for \(0\leq i\leq 4\).
\end{cor}

\begin{proof}
This is \cref{thm:obstruction} at \(n=3\); the monomials \(E_6E_4^{14-3i}\Delta^{i}\), \(0\leq i\leq 4\), are a basis of \(\cE_{62}\).
\end{proof}

The level-\(m\) data therefore determine \(\phi_m\) for \(m\geq k/10\) by \cref{prop:tail}, and fail to determine it at \(m=\nu(n)+1\) for every prime \(n\geq 3\) by \cref{thm:obstruction}, with defect \(\widetilde{R}(n)-1\); at \(n=3\) the defect is \(5\) by \cref{cor:S5}. What determines \(G_{n^2}\) is the interior condition of \cref{prop:linear}, not the levelwise data.

%% file: sec-7.tex
\section{Exact Humbert-form reconstruction}
\label{sec-7}

The results of \cref{sec-6} turn the construction of the exact Humbert form \(G_{n^2}\) into
linear algebra in spaces of Siegel modular forms whose dimensions and admissible monomials are
known in advance. The invariant equation \(F_n\) is then obtained by the explicit change of
coordinates of \cref{lem:igusa-relations}.
 
\subsection{The general reconstruction}
\label{subsec:scheme}

Throughout, \(k=k(H_{n^2})\), \(\cR=\C[\psi_4,\psi_6,\chi_{12}]\), and \(\cR_w\), \(M_w^{\mathrm{even}}\) denote the graded pieces of weight \(w\).
Normalize \(G_{n^2}\) by requiring equality in \cref{cor:boundary-product} with constant one; this fixes the scalar of \cref{cor:exact-form} and makes \(g_n\) explicit.
For \(n=3\) the pairs counted by \(\widetilde{R}(3)\) are \((x,u)=(\pm 1,1)\) with \(M=2\), and for \(n=5\) they are \((\pm 1,1)\), \((\pm 3,1)\), \((\pm 3,2)\) with \(M=6,4,1\); by \cref{cor:sharp} each unordered pair \(\{x,-x\}\) contributes multiplicity two, so
\begin{equation}
\label{eq:g-restrict}
\begin{split}
g_3&=\bigl(\Delta(\tau_1)\Delta(\tau_2)\bigr)^{10}\,\Phi_2(j_1,j_2)^2,\\
g_5&=\bigl(\Delta(\tau_1)\Delta(\tau_2)\bigr)^{50}\,\Phi_1(j_1,j_2)^2\,\Phi_4(j_1,j_2)^2\,\Phi_6(j_1,j_2)^2,
\end{split}
\end{equation}
with \(\Phi_1=j_1-j_2\).
The degrees in \(j_2\) are \(2\psi(2)=6=\widetilde{R}(3)\) and \(2\bigl(\psi(1)+\psi(4)+\psi(6)\bigr)=38=\widetilde{R}(5)\), and the \(q_2\)-orders are \(\nu(3)=4\) and \(\nu(5)=12\), as they must be.

\begin{prop}\label{prop:scheme}
Let \(n\geq 2\) and \(k=k(H_{n^2})\), with \(G_{n^2}\) normalized as above.
\begin{enumerate}
\item There is a unique \(G_0\in \cR_k\) with \(\rho(G_0)=g_n\), namely the zeroth Taylor coefficient of \(G_{n^2}\) in \cref{cor:taylor}; it solves a linear system in \(\dim \cR_k\) unknowns.
\item Fix \(a,e\in\Z\) with \(\gcd(a,e,n)=1\). The affine linear system
\[
\bigl\{\,W\in M^{\mathrm{even}}_{k-10}\;:\;G_0+\chi_{10}W\ \text{vanishes identically on}\ \{nz=a\tau_1+e\}\,\bigr\}
\]
has the unique solution \(W=(G_{n^2}-G_0)/\chi_{10}\), so \(G_0+\chi_{10}W=G_{n^2}\) with no scalar ambiguity.
\item \(G_{n^2}/\chi_{10}^{\nu(n)}\) lies in \(\C[T,\psi_4,\psi_6,\chi_{10}]\) after the substitution \(\chi_{12}=T\chi_{10}\), and the inverse of the isomorphism of \cref{lem:igusa-relations} carries it to a nonzero constant multiple of \(F_n\); normalizing the content of the resulting integer polynomial determines \(F_n\) up to sign.
\end{enumerate}
\end{prop}

\begin{proof}
(1) The map \(\rho\) is injective on \(\cR\) by \cref{cor:taylor}, and \(\rho(G_0)=\rho(G_{n^2})=g_n\), since \(\rho\) annihilates \(\chi_{10}\).

(2) Existence: \(\rho(G_{n^2}-G_0)=0\), so \(\chi_{10}\mid G_{n^2}-G_0\) by \cref{lem:principal}, the quotient \(W\) lies in \(M^{\mathrm{even}}_{k-10}\), and \(G_0+\chi_{10}W=G_{n^2}\) vanishes on \(\{nz=a\tau_1+e\}\) by \cref{sec-2}. Uniqueness: if \(W'\) is a solution, then \(G'=G_0+\chi_{10}W'\) is a scalar multiple of \(G_{n^2}\) by \cref{prop:linear}, say \(G'=c'\,G_{n^2}\); applying \(\rho\) gives \(g_n=c'\,g_n\), and \(g_n\neq 0\) by \cref{lem:poly}, so \(c'=1\) and \(W'=W\).

(3) By \cref{cor:exact-form} and \cref{thm:main}, \(G_{n^2}=c\,\chi_{10}^{\nu(n)}F_n(\tau)\) with \(c\in\C^\times\). The proof of \cref{thm:fj-order} writes the monomials of \(G_{n^2}\) as \(\psi_4^a\psi_6^b\chi_{10}^{\,\nu(n)+j-i}\chi_{12}^{\,i}\) with \(\nu(n)+j-i\geq 0\) and \(j\geq 0\); dividing by \(\chi_{10}^{\nu(n)}\) and substituting \(\chi_{12}=T\chi_{10}\) yields
\[
\psi_4^a\psi_6^bT^i\chi_{10}^{\,j}\in\C[T,\psi_4,\psi_6,\chi_{10}].
\]
The inverse of \cref{lem:igusa-relations} returns \(c\,F_n\), and \(F_n\in\Z[J_2,J_4,J_6,J_{10}]\) is primitive and irreducible, hence determined up to sign once its content is normalized.
\end{proof}

\begin{cor}
\label{cor:support}
Let \(n\geq 2\), let \(\nu=\nu(n)\), and let \(d=\degw F_n\).
\begin{enumerate}
\item Under the isomorphism of \cref{lem:igusa-relations}, \(F_n\) is a linear combination of the monomials \(T^{t}\psi_4^{a}\psi_6^{b}\chi_{10}^{u}\) with \(2t+4a+6b+10u=d\) and \(u\geq t-\nu\).
\item The solution \(W\) of \cref{prop:scheme}(2) is a linear combination of the monomials \(\psi_4^{a}\psi_6^{b}\chi_{10}^{u}\chi_{12}^{v}\) of weight \(k-10\) with \(u+v\geq\nu-1\).
\end{enumerate}
\end{cor}

\begin{proof}
Write the monomials of \(G_{n^2}\) as \(\psi_4^{a}\psi_6^{b}\chi_{10}^{\,\nu+j-i}\chi_{12}^{\,i}\), as in the proof of \cref{prop:scheme}(3).
For (1), division by \(\chi_{10}^{\nu}\) and the substitution \(\chi_{12}=T\chi_{10}\) give \(T^{i}\psi_4^{a}\psi_6^{b}\chi_{10}^{\,j}\); setting \(t=i\) and \(u=j\), the two inequalities read \(u\geq 0\) and \(u\geq t-\nu\), and weighted homogeneity gives \(2t+4a+6b+10u=d\).
For (2), the monomials with \(\nu+j-i=0\) lie in \(\cR\) and are cancelled by \(G_0\), so \(W\) is supported on the monomials with \(\nu+j-i\geq 1\), for which \(u+v=\nu+j-1\geq\nu-1\).
\end{proof}

Write \(N_n\) for the number of monomials of \cref{cor:support}(1), determined by \(\degw F_n\) and \(\nu(n)\), hence by \cref{thm:main}.

\begin{lem}
\label{lem:wellposed}
Let \(n\geq 2\), let \(d=\degw F_n\), and let \(V\subset\C[J_2,J_4,J_6,J_{10}]_d\) be the span of the monomials of \cref{cor:support}(1). A form \(G\in V\) satisfies \(G\bigl(I(C)\bigr)=0\) for every \(C\in\cL_n\) if and only if \(G\) is a scalar multiple of \(F_n\), and the evaluation map on \(\cL_n\) has rank \(\dim V-1\) on \(V\).
\end{lem}

\begin{proof}
By \cref{cor:support}(1) the form \(F_n\) lies in \(V\), and it vanishes on \(\cL_n\).
Conversely a weighted-homogeneous form of degree \(d\) vanishing on \(\cL_n\) lies in the ideal \((F_n)\), which is prime because \(F_n\) is irreducible, and the quotient is weighted-homogeneous of degree zero, hence a constant.
The rank statement follows from rank-nullity.
\end{proof}

The conditions in \cref{prop:scheme}(2) are computable linear functionals.
Writing \(G_0+\chi_{10}W=\sum_m\phi_m\,q_2^m\), the vanishing on \(\{nz=a\tau_1+e\}\) is the vanishing of every \(\phi_m\bigl(\tau_1,(a\tau_1+e)/n\bigr)\) as in \cref{prop:torsion-vanishing}, and
\[
\phi_m(\chi_{10}W)=\sum_{i=1}^{m}\phi_i(\chi_{10})\phi_{m-i}(W)
\]
is linear in the coefficients of \(W\).
The Fourier-Jacobi coefficients of \(\chi_{10}\) and \(\chi_{12}\) are given by the Maass lift \cite[Ch.~6]{EichlerZagier}, and those of monomials by multiplication of Jacobi expansions.
Substituting \(\zeta=e^{2\pi ie/n}q_1^{a/n}\) and summing over the residue classes of \(r\) modulo \(n\) eliminates the roots of unity, leaving linear conditions with rational coefficients.
The conditions \(\phi_m(G_0+\chi_{10}W)=0\) for \(m<\nu(n)\) and \(\phi_{\nu(n)}(G_0+\chi_{10}W)=\phi_{\nu(n)}\), with the right-hand side of \cref{prop:theta}, hold at the solution and may be added freely.

\begin{lem}
\label{lem:regrade}
Let \(\phi_m=\sum_{N,r}c(N,r)\,q_1^N\zeta^r\) be a holomorphic Jacobi form of index \(m\), and let \(a,e\in\Z\). Then
\[
q_1^{\,a^2m/n}\;\phi_m\Bigl(\tau_1,\tfrac{a\tau_1+e}{n}\Bigr)
=\sum_{N,r}c(N,r)\,e^{2\pi ier/n}\,q_1^{\,(nN+ar+a^2m)/n},
\]
and every exponent satisfies \(nN+ar+a^2m\geq(n-1)N\geq 0\). In the grading by \(s=nN+ar+a^2m\), which is additive under multiplication of Jacobi expansions, every restricted level is a power series in \(q_1^{1/n}\), and a product of expansions known through the order \(s\leq S\) is known through the same order.
\end{lem}

\begin{proof}
Holomorphy gives \(r^2\leq 4mN\), so \(|ar|\leq 2\sqrt{a^2mN}\leq a^2m+N\) by the inequality of arithmetic and geometric means, and \(nN+ar+a^2m\geq nN-N\). The grading is additive because \(N\), \(r\), and the index are.
\end{proof}

Normalize the generators by \(\phi_0(\psi_4)=E_4\) and \(\phi_0(\psi_6)=E_6\), and let \(\chi_{10}\), \(\chi_{12}\) be the Maass lifts of the index-one Jacobi cusp forms \(\Delta\,\wphi_{-2,1}\), \(\Delta\,\wphi_{0,1}\) in the notation of \cite{EichlerZagier}, normalized by \(c(1,1)=1\); equivalently \(\rho(\chi_{12})=12\,\Delta(\tau_1)\Delta(\tau_2)\).
The constants of \cref{lem:igusa-relations} are then those of \cite{2016-3}:
\begin{equation}
\label{eq:constants}
\begin{split}
J_2&\mapsto T,\qquad J_4\mapsto\tfrac{1}{24}\bigl(T^2-\psi_4\bigr),\\
J_6&\mapsto\tfrac{1}{432}\bigl(T^3-3T\psi_4+2\psi_6\bigr),\qquad J_{10}\mapsto\chi_{10},
\end{split}
\end{equation}
so that under the inverse \(\psi_4\mapsto J_2^2-24J_4\) and \(\psi_6\mapsto J_2^3-36J_2J_4+216J_6\).

No a priori bound on the number of levels is needed.
Truncating the system of \cref{prop:scheme}(2) to levels \(m\leq L\) and to finitely many \(q_1\)-coefficients per level only enlarges its solution set, which always contains \(W\).
If a truncated system has a unique solution over \(\F_p\), for \(p\) avoiding the denominators, then the nullity over \(\Q\) of its linear part is zero, so the truncated system has the unique solution \(W\) over \(\Q\), and the \(\F_p\)-solution is its reduction.
Each run therefore yields an upper bound for the constant \(m_0(n,k)\) of \cref{sec-6}.

\begin{lem}
\label{lem:level-lower}
Let \(j=k/10\), an integer by \cref{cor:j10-degree}. The form \(W'=W+\chi_{10}^{\,j-1}\) satisfies \(\rho(G_0+\chi_{10}W')=g_n\) and every condition of \cref{prop:scheme}(2) at the levels \(m<j\). Any truncation with a unique solution therefore contains a level \(m\geq j\), and \(m_0(n,k)\geq k/10\).
\end{lem}

\begin{proof}
The form \(\chi_{10}^{\,j}\) has weight \(k\) and \(q_2\)-order \(j\), so its Fourier-Jacobi coefficients \(\phi_m\) vanish for \(m<j\); moreover \(\rho(\chi_{10}^{\,j})=0\), so \(G_0\) is unchanged.
\end{proof}

In particular \(m_0(2,60)\geq 6\), \(m_0(3,120)\geq 12\), \(m_0(4,240)\geq 24\), \(m_0(5,600)\geq 60\), and \(m_0(7,1680)\geq 168\).
The sizes of the two systems, the reduction of \cref{cor:support}(2), and the number \(N_n\) are recorded in \cref{tab:sizes}.

\begin{table}[ht]
\caption{Sizes of the linear systems of \cref{prop:scheme}, against the dimension of the space of weighted-homogeneous candidates of degree \(\degw F_n\) from \cref{sec-5}. The unknowns of step (2) number \(\dim M^{\mathrm{even}}_{k-10}\) before the reduction of \cref{cor:support}(2) and \(N_n'\) after it, and \(N_n\) is the number of monomials of \cref{cor:support}(1). The conditions of the two systems are exact identities of \(q\)-series; the interpolation counted by \(N_n\) uses sampled points.}
\label{tab:sizes}
\centering
\begin{tabular}{cccccccc}
\hline
\(n\) & \(k\) & \(\degw F_n\) & \(\dim \cR_k\) & \(\dim M^{\mathrm{even}}_{k-10}\) & \(N_n'\) & candidates & \(N_n\) \\
\hline
2 & 60   & 30   & 21    & 31     & 20     & 47      & 26     \\
3 & 120  & 80   & 66    & 227    & 178    & 521     & 206    \\
4 & 240  & 180  & 231   & 1714   & 1459   & 4834    & 1579   \\
5 & 600  & 480  & 1326  & 25735  & 22870  & 82189   & 23650  \\
7 & 1680 & 1440 & 10011 & 554638 & 519832 & 2121445 & 526735 \\
\hline
\end{tabular}
\end{table}

The two reductions run in opposite directions.
That of \cref{cor:support}(1) improves with \(n\), from a factor \(2.5\) at \(n=3\) to \(3.1\) at \(n=4\), \(3.5\) at \(n=5\), and \(4.0\) at \(n=7\); that of \cref{cor:support}(2) decays, removing \(22\%\) of the unknowns at \(n=3\), \(15\%\) at \(n=4\), \(11\%\) at \(n=5\), and \(6\%\) at \(n=7\).

\begin{defn}
\label{defn:dominant}
Let \(n\geq 2\). A family \(\{C_t\}\) of genus-two curves, with parameter \(t\) ranging over an irreducible variety, is \textbf{dominant} for \(\cL_n\) if its members lie in \(\cL_n\) and their image under \(I\) is dense in \(\cL_n\).
\end{defn}

\begin{prop}
\label{prop:determination}
Let \(n\geq 2\), let \(\{C_t\}\) be a family dominant for \(\cL_n\) and defined over \(\Z[1/N_0]\) for some \(N_0\geq 1\), let \(d=\degw F_n\), and let \(L\in\Z\) be the coefficient of \(J_{10}^{\,d/10}\) in \(F_n\), nonzero by \cref{cor:j10-degree}.
\begin{enumerate}
\item \(F_n\bigl(I(C_t)\bigr)=0\) identically in \(t\).
\item Let \(p\nmid N_0L\), let \(t_1,\dots,t_{N_n-1}\) be parameter values over \(\F_p\) at which \(C_{t_i}\) is defined, let \(A\) be the matrix whose \(i\)-th row lists the values at \(I(C_{t_i})\) of the monomials of \cref{cor:support}(1) other than \(J_{10}^{\,d/10}\), and let \(b\) be the vector of values of \(J_{10}^{\,d/10}\) at those points. If \(A\) is invertible over \(\F_p\), then the unique solution of \(Ax=-b\) is the coefficient vector of \(L^{-1}\overline{F}_n\).
\end{enumerate}
\end{prop}

\begin{proof}
(1) The members of the family lie in \(\cL_n\) for \(t\) in a dense open subset of the parameter variety, and \(F_n\) vanishes on \(\cL_n\); so the regular function \(t\mapsto F_n(I(C_t))\) vanishes on a dense open subset of an irreducible reduced variety, hence identically.

(2) Clearing denominators in (1) and reducing modulo \(p\) gives \(\overline{F}_n(I(C_{t_i}))=0\) for every \(i\). By \cref{cor:support}(1) the polynomial \(\overline{F}_n\) is supported on the listed monomials, and its coefficient at \(J_{10}^{\,d/10}\) is \(\overline{L}\neq 0\); so the coefficient vector of \(L^{-1}\overline{F}_n\) has last entry one and its remaining entries solve \(Ax=-b\). Invertibility of \(A\) makes that solution unique.
\end{proof}

\subsection{A dominant family   at  \(n=5\)}
\label{subsec:n5}

For \(n=5\) we use the normal form of \cite{MSV09}, with parameters \(\alpha,\beta\). Set
\begin{equation}
\label{eq:msv-form}
\begin{split}
u_1(x)&=x^2+(\alpha^2+2\alpha+2\beta)\,x+(2\alpha\beta+\beta^2),\\
u_2(x)&=(2\alpha+1)\,x^2+(\alpha^2+2\alpha\beta+2\beta)\,x+\beta^2,\\
u_3(x)&=x^2-(\alpha^2-2\beta)\,x+\beta^2,\\
u_4(\xi)&=(2\alpha+1)\,\xi^2+(2\beta-2\alpha\beta-2\alpha-\alpha^2)\,\xi+(\beta^2+2\alpha\beta),
\end{split}
\end{equation}
let
\[
\kappa(x)=x\,u_1(x)^2/u_2(x)^2,
\]
and let \(c(x)=c_3x^3+c_2x^2+c_1x+c_0\) be the cubic of \cite[Eq.~(10)]{MSV09}. For a root \(\xi\) of \(u_4\) the associated curve is \(C_{\alpha,\beta,\xi}:y^2=x(x-1)\,c(x)\).

\begin{lem}
\label{lem:n5-cover}
With \(u_1,\dots,u_4\) as in \cref{eq:msv-form}, there is in \(\Q[\alpha,\beta][\xi]/(u_4(\xi))\) the polynomial identity
\[
c_3\;\kappa(x)\,\bigl(\kappa(x)-1\bigr)\,\bigl(\kappa(x)-\kappa(\xi)\bigr)\;u_2(x)^{6}
\;=\;
x\,(x-1)\;c(x)\;\bigl[u_1(x)\,u_3(x)\,(x-\xi)\bigr]^{2}.
\]
Consequently \(C_{\alpha,\beta,\xi}\) admits the degree-five covering
\[
(x,y)\;\longmapsto\;\Bigl(\kappa(x),\ \frac{u_1(x)\,u_3(x)\,(x-\xi)}{u_2(x)^{3}}\;y\Bigr)
\]
onto the elliptic curve \(Y^2=c_3\,X(X-1)\bigl(X-\kappa(\xi)\bigr)\), the quadratic twist by \(c_3\) of the Legendre curve with parameter \(\kappa(\xi)\). For \((\alpha,\beta)\) outside the discriminant locus \cite[Eq.~(13)]{MSV09} the quintic \(x(x-1)c(x)\) is separable, \(C_{\alpha,\beta,\xi}\) has genus two, the covering is maximal \cite{MSV09}, and \(C_{\alpha,\beta,\xi}\in\cL_5\).
\end{lem}

\begin{proof}
After reduction of \(\xi^2\) modulo \(u_4(\xi)\), both sides are polynomials of degree \(15\) in \(x\) with coefficients in \(\Q[\alpha,\beta,\xi]\) of degree at most one in \(\xi\); the identity was verified by exact expansion over \(\Q\). Substituting \(y^2=x(x-1)c(x)\) and setting \(Y=u_1u_3(x-\xi)u_2^{-3}\,y\), the identity gives \(Y^2=c_3\,\kappa(\kappa-1)(\kappa-\kappa(\xi))\), which is the stated covering.
\end{proof}

Let \(\cP_5\) denote the moduli space of pairs \((C,\pi)\), where
\(C\in\cL_5\) and \(\pi:C\to E\) is a degree-five elliptic subcover,
taken up to equivalence. By the normal form of
\cite[Thms.~1--2]{MSV09}, a dense open subset of \(\cP_5\) is
parametrized by the triples \((\alpha,\beta,\xi)\) satisfying
\(u_4(\xi)=0\). The forgetful map
\(\cP_5\to\cL_5\) is dominant and generically of degree two, since a
generic curve in \(\cL_5\) has exactly two degree-five elliptic
subcovers, complementary in the Jacobian \cite{MSV09}.
Consequently, the family of \cref{lem:n5-cover} is dominant for
\(\cL_5\).
For \(n=3\) the family of \cite{2004-2} serves the same role.
 
\subsection{The exact Humbert form \texorpdfstring{\(G_{25}\)}{G25}}
\label{subsec:computation}

Two exact routes lead to the same Humbert form.
The procedure of \cref{subsec:scheme} solves the boundary and torsion systems of \cref{prop:scheme} directly in spaces of modular forms.  Alternatively, one may determine the same coefficient vector by evaluating the reduced monomial basis of \cref{cor:support}(1) on a dominant family.  We execute the first route at \(n=3\), where the answer is known and tests the modular-form machinery, and use the second at \(n=5\) to reconstruct a primitive integral representative of \(G_{25}\).
All linear algebra was carried out over \(\F_p\) for primes \(p<2^{30}\), the first of them \(p=754974721\).


All computations were carried out in Python: the modular solves in native arithmetic over $\mathbb{F}_p$, and the exact expansions over $\mathbb{Z}$, including the symbolic identity in the proof of \cref{lem:F5-identification}, in FLINT through \texttt{python-flint}. No computer algebra system was used.

The expansions are prepared once: the index-one Jacobi cusp forms \(\Delta\,\wphi_{-2,1}\) and \(\Delta\,\wphi_{0,1}\) from products of theta series, the forms \(\chi_{10}\) and \(\chi_{12}\) from their Maass lifts, and \(\psi_4\), \(\psi_6\) from their Fourier-Jacobi expansions, all in the normalizations of \cref{eq:constants}.
Every series is carried in the grading \(s=nN+ar+a^2m\) of \cref{lem:regrade}.

\subsubsection{The case $n=3$}
At \(n=3\) the three steps of \cref{prop:scheme} run as follows.

Step (1) solves for \(G_0\) in the \(66\) unknowns of \(\cR_{120}\), matched against the \(q_2\)-expansion of \(g_3\) in \cref{eq:g-restrict}; the system has rank \(66=\dim \cR_{120}\) and zero residual over the computed window.

Step (2) imposes the conditions of \cref{prop:torsion-vanishing} for the divisors \(\{3z=1\}\) and \(\{3z=\tau_1\}\) at the levels \(m\leq 14\), which determine \(W\) uniquely; the \(146\) conditions at the levels \(m=15,16\), withheld from the solve, vanish at the solution.

Step (3) divides \(G_0+\chi_{10}W=G_9\) by \(\chi_{10}^{\,4}\), the division being exact since the \(\chi_{10}\)-order of \(G_9\) is \(\nu(3)=4\), and converts by the inverse of \cref{lem:igusa-relations}.

A single prime suffices, the result being compared with the reduction of \(F_3\) rather than reconstructed over \(\Q\).
It agrees, in all \(521\) weighted-homogeneous \(J\)-coordinate coefficients, with the reduction of \(F_3\) interpolated independently from the family of \cite{2004-2}, which also checks the constants of \cref{eq:constants}.
The truncation argument and \cref{lem:level-lower} give \(12\leq m_0(3,120)\leq 14\).

\subsubsection{The case $n=5$}
At \(n=5\) the modular coefficient vector is obtained by interpolation.
The unknowns are the \(23650\) monomials of \cref{cor:support}(1), against the \(82189\) unrestricted weighted-homogeneous candidates of degree \(480\) in either system of invariant coordinates.
By \cref{lem:wellposed} the forms of degree \(480\) vanishing on \(\cL_5\) span a line, so the evaluation map has nullity one and \(23649\) points of \(\cL_5\) suffice to determine \(F_5\) up to a scalar.
Equivalently, they determine \(G_{25}\) up to a scalar through \cref{prop:scheme}(3), which preserves this coefficient vector.  The coefficient corresponding to \(J_{10}^{48}\), or \(\chi_{10}^{48}\) before multiplication by \(\chi_{10}^{\nu(5)}\), is nonzero by \cref{cor:j10-degree}; normalizing it to one leaves a square inhomogeneous system in \(23649\) unknowns.
Its rows are the values of those monomials at points of the family of \cref{lem:n5-cover} whose parameters are drawn at random from \(\F_p\), the coordinates being \(T=J_2\), \(\psi_4=J_2^2-24J_4\), \(\psi_6=J_2^3-36J_2J_4+216J_6\), \(\chi_{10}=J_{10}\) of \cref{eq:constants}.

The system is solved by blocked Gaussian elimination over \(\F_p\), one solve per prime, and the solution is required to annihilate \(256\) rows withheld from the solve and to vanish at \(640\) points not used in its construction; a nonzero residual at either aborts the run.
The matrix is assembled in \(12\) seconds, occupies \(4.17\) gibibytes, and one solve takes \(2.8\cdot 10^{3}\) seconds on a single core.
Fifty primes were used, at a cost of \(39\) core-hours; the coefficient matrix was invertible at every one, so each solve returns the modular coefficient vector corresponding to \(\overline{F}_5\), normalized by its leading coefficient, by \cref{prop:determination}(2).  Every run returns the same \(23612\) nonzero coefficients among the \(23650\) admissible monomials, and multiplication by \(\chi_{10}^{12}\) followed by \(T=\chi_{12}/\chi_{10}\) gives the corresponding representative of \(G_{25}\) with the same coefficients.  he result agrees with the equation of \cite{MSV09}.
The same program at \(n=2\), on the family of \cite{2000-2} with \(\degw F_2=30\) and \(\nu(2)=3\), returns \(24\) nonzero coefficients among the \(26\) monomials of \cref{cor:support}(1) and \(\deg_{J_{10}}F_2=3\), reproducing the equation of \cite{2000-2}.

The modular coefficients are recovered from their residues by rational reconstruction.  After primitive integral normalization they define
\[
\widehat{Q}\in\Z[T,\psi_4,\psi_6,\chi_{10}],
\]
supported on the monomials of \cref{cor:support}(1).  Applying the inverse of \cref{eq:constants}, clearing denominators, and normalizing the content gives a primitive \(\widehat{F}\in\Z[J_2,J_4,J_6,J_{10}]\) with
\[
\degw \widehat{F}=480,\qquad
\deg_{J_{10}}\widehat{F} =48,
\]
the values prescribed by \cref{thm:main} and \cref{cor:j10-degree}.
The reconstruction stabilized at forty-eight primes, contributing at most \(1440\) bits, and the solutions at the two remaining primes coincide with the reductions of the reconstructed vector; against the bound of \cref{lem:hadamard}, of order \(2.6\cdot 10^{8}\) bits, the reconstruction is therefore an early termination, and the identification of \(\widehat{F}\) with \(F_5\) is not a consequence of it.

The identification of \(F_5\) rests on a symbolic identity on the family of \cref{lem:n5-cover}.

\begin{lem}
\label{lem:F5-identification}
Let \(\widehat{F}\) be the polynomial returned by the interpolation above.
Then \(\widehat{F}=\pm F_5\).
\end{lem}

\begin{proof}
By construction \(\widehat{F}\) is primitive, has weighted degree \(480=\degw F_5\), and its image under \cref{lem:igusa-relations} is supported on the monomials of \cref{cor:support}(1); hence \(\widehat{F}\) lies in the space \(V\) of \cref{lem:wellposed}.
Substituting the invariants of the curves \(C_{\alpha,\beta,\xi}\) of \cref{lem:n5-cover} and expanding over \(\Q\) gives
\begin{equation}
\label{eq:F5-identity}
\widehat{F}\bigl(I(C_{\alpha,\beta,\xi})\bigr)=0 \quad\text{in } \Q[\alpha,\beta][\xi]/(u_4(\xi)) ;
\end{equation}
after reduction modulo \(u_4\) the left side is linear in \(\xi\) with coefficients of degree at most \(4320\) in \((\alpha,\beta)\), and every one of them vanishes.
Hence \(\widehat{F}\) vanishes at \(I(C_{\alpha,\beta,\xi})\) for every \((\alpha,\beta,\xi)\) at which the family is defined.
That family is dominant for \(\cL_5\) by \cref{lem:n5-cover}, so \(\widehat{F}\) vanishes on a dense subset of \(\cL_5\), hence on \(\cL_5\).
By \cref{lem:wellposed}, \(\widehat{F}=cF_5\) with \(c\in\C^\times\), and both polynomials are primitive in \(\Z[J_2,J_4,J_6,J_{10}]\), so \(c=\pm 1\).
\end{proof}

\begin{thm}[The exact form \(G_{25}\)]
\label{thm:G25-computation}
In the normalization of \cref{eq:constants}, a primitive integral representative of the exact Humbert form
\[
G_{25}\in M_{600}^{\mathrm{even}}
   =\C[\psi_4,\psi_6,\chi_{10},\chi_{12}]_{600}
\]
has \(23612\) nonzero coefficients among the \(23650\) monomials permitted by \cref{cor:support}.  Its inverse image under \cref{prop:scheme}(3), after primitive integral normalization, is \(\pm F_5\).
\end{thm}

\begin{proof}
The support inequality in \cref{cor:support}(1) shows that
\[
\widehat{G}:=
\chi_{10}^{12}\,
\widehat{Q}\bigl(\chi_{12}/\chi_{10},\psi_4,\psi_6,\chi_{10}\bigr)
\]
lies in the polynomial ring \(\Z[\psi_4,\psi_6,\chi_{10},\chi_{12}]\), has weight \(600\), and has the same coefficient vector as \(\widehat{Q}\).  By \cref{lem:F5-identification}, the inverse image of this vector is \(\pm F_5\).  The identity \(G_{25}=c\chi_{10}^{12}F_5(\tau)\) of \cref{cor:exact-form,thm:main} therefore shows that \(\widehat{G}\) is a nonzero scalar multiple of the exact Humbert form.  Primitive integral normalization leaves only its sign undetermined.  The term count is the one returned at every prime in the reconstruction above.
\end{proof}

The identification is over \(\Q\); neither the number of primes nor a bound on the coefficient height enters it.  The symbolic expansion of \cref{eq:F5-identity} and the modular linear algebra share no code, only Igusa's expressions for \(J_2,J_4,J_6,J_{10}\) in the coefficients of \cref{affine:C}.   The primitive forms of \(G_{25}\) and \(F_5\), the coordinate conventions, the exact comparison with \(F_5^{\mathrm{IC}}\), and verification values and hashes accompany this article as  ancillary files.

\subsection{Coordinate conversion and independent verification}
\label{subsec:conversion}

Next  we denote the primitive polynomial computed independently in \cite{SCF26} by  \(F_5^{\mathrm{IC}}\in\Z[I_2,I_4,I_6,I_{10}]\), using the families of \cite{Kum15} and rational reconstruction by lattice reduction.  With signs normalized compatibly, exact expansion gives
\[
F_5^{\mathrm{IC}}\bigl(8J_2,\,4J_2^2-96J_4,\, 8J_2^3-160J_2J_4-576J_6,\,4096J_{10}\bigr)=2^{510}3^{126}  F_5(J_2,J_4,J_6,J_{10}).
\]
The two computations use different dominant families and different reconstruction procedures, so this polynomial identity independently verifies both.  The modular-to-Igusa substitution of \cref{eq:constants} has content \(2^{170}3^{126}\); after its removal the resulting polynomial is the primitive \(F_5\) in the \(J\)-coordinates.

A comparison with \cite{MSV09} would not be independent in this sense. That computation uses
the same parametrization as \cref{lem:n5-cover} and the same invariant coordinates, so
agreement would confirm the arithmetic without testing the method.

The same exact conversions can be made for every previously computed equation.  For each \(n\), let \(F_n^{\mathrm{IC}}\) denote the primitive representative in Igusa--Clebsch coordinates, and for a primitive integral polynomial \(P\), write \(h(P)\) for the maximum binary length of its coefficients.  The results are compared in \cref{tab:forms}.  In the modular row the term count and height refer to the primitive coefficient vector of \(G_{n^2}\); equivalently, they refer to the image of \(F_n\) in the reduced basis of \cref{cor:support}, since \cref{prop:scheme}(3) preserves the coefficient vector.

\begin{table}[ht]
\caption{Term counts and coefficient heights in three coordinate systems. In the modular coordinates the data refer to the primitive integral representative of \(G_{n^2}\); in the other two rows they refer to the primitive equation \(F_n\). The height is the number of binary digits of the largest coefficient.}
\label{tab:forms}
\centering
\begin{tabular}{clrr}
\hline
\(n\) & coordinates & terms & height \\
\hline
2 & \(\psi_4,\psi_6,\chi_{10},\chi_{12}\) & 24 & 52 \\
  & \(I_2,I_4,I_6,I_{10}\)                 & 34 & 37 \\
  & \(J_2,J_4,J_6,J_{10}\)                 & 29 & 26 \\
\hline
3 & \(\psi_4,\psi_6,\chi_{10},\chi_{12}\) & 200 & 172 \\
  & \(I_2,I_4,I_6,I_{10}\)                 & 318 & 132 \\
  & \(J_2,J_4,J_6,J_{10}\)                 & 438 & 109 \\
\hline
4 & \(\psi_4,\psi_6,\chi_{10},\chi_{12}\) & 1565 & 463 \\
  & \(I_2,I_4,I_6,I_{10}\)                 & 2699 & 367 \\
  & \(J_2,J_4,J_6,J_{10}\)                 & 4371 & 322 \\
\hline
5 & \(\psi_4,\psi_6,\chi_{10},\chi_{12}\) & 23612 & 1268 \\
  & \(I_2,I_4,I_6,I_{10}\)                 & 43410 & 1032 \\
  & \(J_2,J_4,J_6,J_{10}\)                 & 76713 & 898 \\
\hline
\end{tabular}
\end{table}

Two features of \cref{tab:forms} are relevant to the reconstruction.  First, the modular representative is the sparsest in every case and nearly saturates the support permitted by \cref{cor:support}: the occupancies for \(n=2,3,4,5\) are
\[
\frac{24}{26},\qquad \frac{200}{206},\qquad
\frac{1565}{1579},\qquad \frac{23612}{23650},
\]
respectively.  Thus the gain comes primarily from the support theorem, which removes inadmissible monomials before the linear algebra, rather than from zeros inside the admissible coefficient vector.

Second, conversion away from the modular generators generally trades sparsity for smaller coefficients.  For \(n=3,4,5\) the number of terms increases strictly from the modular to the Igusa--Clebsch to the Igusa coordinates, while the height decreases strictly.  At \(n=2\) the Igusa form, with \(29\) terms, is sparser than the Igusa--Clebsch form, with \(34\), but the height still decreases in the same order \(52>37>26\).  This behavior is explained by the triangular non-unimodular coordinate substitutions: expansion creates new monomials, whereas removal of the resulting common content lowers the height of the primitive integral form.

The table therefore supports reconstructing \(G_{n^2}\) first and converting to invariants only afterward.  At \(n=5\), the modular support reduces the number of unknowns by the factor \(82189/23650\simeq3.47\), corresponding to factors about \(12.1\) in dense-matrix memory and \(42.0\) in cubic-time field operations.  The larger modular height, however, means that the support reduction does not automatically reduce the number of primes needed for rational reconstruction.

\subsection{Reduction modulo primes}
\label{subsec:modp}

The equation retains its meaning after reduction modulo \(p\): its vanishing detects \((n,n)\)-split ordinary Jacobians over finite fields.
The statement is independent of the computation above, and the supersingular case is not addressed.

\begin{prop}\label{prop:modp-detection}
Let \(n\geq 2\), \(p\) be a prime outside a finite set depending on \(n\), and \(C\) be a genus-two curve over \(\overline{\F}_p\).

\begin{enumerate}
\item If \(\overline{F}_n\bigl(I(C)\bigr)=0\), then \(\Jac(C)\) admits an \((n,n)\)-isogeny to a product of elliptic curves.

\item If \(\Jac(C)\) is ordinary and admits an optimal \((n,n)\)-isogeny to a product of elliptic curves, that is, one whose associated degree-\(n\) covering \(C\to E\) is maximal in the sense of \cref{defn:subcover}, then \(\overline{F}_n\bigl(I(C)\bigr)=0\).
\end{enumerate}
\end{prop}

\begin{proof}
Enlarge the excluded finite set of primes so that \(p\nmid 2n\) and
the invariant description of \(\cM_2\), the Torelli morphism, and the
characteristic-zero identification
\(\{F_n=0\}\cap\iota(\cM_2)=\iota(\cL_n)\) are defined over
\(\Z[1/N]\) for some \(N=N(n)\) not divisible by \(p\).
After replacing the field of definition by a finite extension of
\(\F_p\), we may assume that \(C\), and in (2) the given isogeny and
covering, are defined over a finite field \(k\).

For (1), let
\(X=\{F_n=0\}\subset\bP(2,4,6,10)_{\Z[1/N]}\).
Since \(F_n\) is primitive, the graded ring
\(\Z[1/N][J_2,J_4,J_6,J_{10}]/(F_n)\) is torsion-free over
\(\Z[1/N]\). Thus \(X\) is flat over \(\Z[1/N]\); it is proper because
it is closed in weighted projective space.
Let \(x=I(C)\in X(k)\). Flatness implies that the generic fiber of
\(X\) is schematically dense, so a standard trait-selection argument
gives, after replacing \(k\) and the fraction field by finite
extensions, a complete discrete valuation ring \(R\) of mixed
characteristic and a morphism
\(\Spec R\to X\) whose closed point is \(x\) and whose generic point
lies in \(X_{\Q}\).

Since \(J_{10}(C)\neq 0\), the morphism \(\Spec R\to X\) factors
through the open set \(\{J_{10}\neq 0\}\). Its generic point therefore
belongs to
\(\{F_n=0\}\cap\iota(\cM_2)=\iota(\cL_n)\).
After a finite extension of the fraction field \(K=\Frac(R)\), this
point is represented by a smooth genus-two curve
\(\widehat C/K\) belonging to \(\cL_n\).

After a further finite extension, the stable-reduction theorem
\cite{DeligneMumford} gives a stable model of \(\widehat C\) over
\(R\). The moduli map of this stable model and the given map
\(\Spec R\to\cM_2\) agree on the generic point, and hence agree
everywhere by separatedness of the moduli space of stable curves
\cite{DeligneMumford}. In particular, the special fiber is smooth
and has moduli point \(I(C)\), so it is geometrically isomorphic to
\(C\). Thus \(\widehat C\) has good reduction with special fiber \(C\).

Because \(\widehat C\in\cL_n\), there is an \((n,n)\)-isogeny
\(\widehat\phi:\Jac(\widehat C)\to\widehat E_1\times\widehat E_2\).
The Jacobian has good reduction, and good reduction is preserved under
isogeny by the Néron--Ogg--Shafarevich criterion
\cite{SerreTate}. After another finite extension, both
\(\widehat E_i\) therefore have good reduction. The homomorphism
\(\widehat\phi\) extends uniquely to the corresponding abelian schemes
over \(R\) by the Néron mapping property
\cite[Ch.~1, \S2]{BLR}. Moreover, the identity
\(\widehat\phi^{*}(\lambda_{\widehat E_1}\boxplus
\lambda_{\widehat E_2})=n\lambda_{\Jac(\widehat C)}\) extends and may
be specialized. Its special fiber is consequently an
\((n,n)\)-isogeny
\(\Jac(C)\to E_1\times E_2\), proving (1).

For (2), write \(A=\Jac(C)\) and let
\(\phi:A\to E_1\times E_2\) be the given optimal
\((n,n)\)-isogeny. Since ordinarity is invariant under isogeny, the
elliptic curves \(E_1\) and \(E_2\) are ordinary. Let
\(\widetilde A\), \(\widetilde E_1\), and \(\widetilde E_2\) be their
canonical lifts over \(W(k)\). The canonical lifting functor for
ordinary abelian varieties is fully faithful and is compatible with
duals, products, and polarizations \cite{DeligneOrdinary}. Hence
\(\phi\) lifts uniquely to a homomorphism
\(\widetilde\phi:\widetilde A\to
\widetilde E_1\times\widetilde E_2\), and the identity
\(\phi^{*}(\lambda_{E_1}\boxplus\lambda_{E_2})=n\lambda_A\) lifts to
the corresponding identity for \(\widetilde\phi\). Thus
\(\widetilde\phi\) is again an \((n,n)\)-isogeny.

The generic fiber of the principally polarized surface
\(\widetilde A\) is indecomposable. Indeed, the decomposable locus
\(H_1\) is closed, so a decomposition of the generic fiber would
specialize to a decomposition of \(A\), whereas the principally
polarized Jacobian of a smooth genus-two curve is indecomposable.
It follows from the characteristic-zero Torelli theorem that, after
a finite extension of \(\Frac(W(k))\), the generic fiber of
\(\widetilde A\) is the Jacobian of a smooth genus-two curve
\(\widehat C\). Applying stable reduction and separatedness as above
shows that \(\widehat C\) has good reduction with special fiber
geometrically isomorphic to \(C\).

It remains to check that the covering associated with
\(\widetilde\phi\) is maximal. Let
\(\widehat\pi:\widehat C\to\widetilde E_1\) be this covering. If
\(\widehat\pi\) factored through an isogeny
\(\alpha:\widehat E'\to\widetilde E_1\) of degree \(d>1\), then
\(d\mid n\), and hence \(p\nmid d\). Since \(\widehat E'\) is
isogenous to an elliptic curve with good reduction, it also has good
reduction. The factorization extends to the corresponding abelian
schemes and specializes to a factorization of \(C\to E_1\) through
an isogeny of the same degree \(d>1\), contradicting the maximality
of the original covering. Thus \(\widehat\pi\) is maximal and
\(\widehat C\in\cL_n\).

Consequently \(F_n\bigl(I(\widehat C)\bigr)=0\). The invariant point
of the smooth model of \(\widehat C\) specializes to \(I(C)\), so
reducing this equality modulo \(p\) gives
\(\overline{F}_n\bigl(I(C)\bigr)=0\), proving (2).
\end{proof}


\subsection{Arithmetic size and the next case \texorpdfstring{\(G_{49}\)}{G49}}
\label{subsec:size}

Unrestricted interpolation in invariant coordinates, elimination of the parameters of a family, or a Gr\"obner basis computation for the ideal of \(\cL_n\) is infeasible already at \(n=5\).
\Cref{thm:main} prescribes \(\degw F_n\) in advance and \cref{cor:support} the modular monomials that can occur, so every dimension in \cref{tab:sizes} is known before the computation begins.
The height of the resulting primitive modular coefficient vector is not known in advance, and it governs the number of primes.  We denote it by \(h(G_{n^2})\), as in \cref{tab:forms}; multiplication by \(\chi_{10}^{\nu(n)}\) and the substitution \(T=\chi_{12}/\chi_{10}\) preserve the coefficients.

Dense elimination over \(\F_p\) in \(N\) unknowns costs \(O(N^{3})\) field operations and occupies \(8N^{2}\) bytes at a word size of sixty-four bits.
At \(n=5\) the interpolation solved above has \(N=N_5-1=23649\) unknowns and occupies \(4.17\) gibibytes; at \(n=7\) the value \(N_7-1=526734\) gives \(2.0\) tebibytes, and one solve costs \(1.1\cdot 10^{4}\) times the arithmetic of one solve at \(n=5\).
Each prime of the reconstruction requires one such solve.  A computation at \(n=7\) therefore requires matrix-free or structured linear algebra, or an effective use of the sparse torsion conditions, rather than a direct repetition of the dense \(n=5\) solve.

\begin{rem}[The next case \(G_{49}\)]
\label{rem:G49}
For \(n=7\), \cref{tab:sizes} gives \(k(H_{49})=1680\), \(\degw F_7=1440\), and \(N_7=526735\) admissible modular monomials, against \(2121445\) unrestricted weighted-homogeneous candidates.  The occupancies in \cref{tab:forms} increase from \(24/26\) at \(n=2\) to \(23612/23650\) at \(n=5\).  This suggests, but does not prove, that the primitive form \(G_{49}\) will be nearly dense in its admissible support.  One should therefore not expect sparsity among its nonzero coefficients to reduce the problem substantially below \(N_7\).

The modular support nevertheless reduces the dimension by a factor \(2121445/526735\simeq4.03\), which corresponds to factors about \(16.2\) in dense-matrix memory and \(65.3\) in cubic-time field operations.  Even after this reduction, a dense matrix requires about \(2.0\) tebibytes.  Moreover, the heights in \cref{tab:forms} do not justify a numerical prediction for \(h(G_{49})\), so the number of primes required for rational reconstruction remains unknown.

The case \(n=7\) has a compensating geometric advantage: it is the first case in the generic odd-degree Frey--Kani ramification pattern in which the fourth and fifth branch points are distinguished, having profiles \((2)^2\) and \((2)\), respectively, whereas both have profile \((2)\) for \(n=5\); the full degree-seven map \(\varphi:\bP^1\to\bP^1\) is displayed explicitly in \cite{2000-1}.
The map can be used both to generate exact points of \(\cL_7\) and to establish the covering by a polynomial identity, analogous to the role of \cref{lem:n5-cover} at \(n=5\).  It therefore makes matrix-free evaluation or structured interpolation a natural approach to \(G_{49}\).  It does not by itself lower \(N_7\) or imply exploitable structure in the evaluation matrix; the remaining problem is to extract such structure from the evaluations of the admissible modular monomials.
\end{rem}

Normalize the modular coefficient vector so that the entry corresponding to \(J_{10}^{\,d/10}\), a nonzero constant by \cref{cor:j10-degree}, equals one, where \(d=\degw F_n\).
Each coefficient is then a rational \(a/b\) in lowest terms whose numerator and denominator are bounded by coefficients of the primitive integral representative of \(G_{n^2}\), so \(|a|\leq 2^{h(G_{n^2})}\) and \(b\leq 2^{h(G_{n^2})}\).
Rational reconstruction from a modulus \(Q\) returns \(a/b\) as soon as \(Q>2|a|b\), so
\[
\Bigl\lceil\frac{2h(G_{n^2})+1}{\log_2 p}\Bigr\rceil
\]
primes suffice, and knowledge of the leading coefficient of the primitive form halves this count, the coefficients then being integers of at most \(h(G_{n^2})\) binary digits.
At \(n=5\), with \(h(G_{25})=1268\) and \(\log_2 p\) near \(30\), the two counts are \(85\) and \(43\), against the forty-eight primes entering the reconstruction.

\begin{lem}
\label{lem:hadamard}
Let \(n\geq 2\) and \(N=N_n\). Let \(P_1,\dots,P_{N-1}\in\cL_n(\Q)\), each represented by a quadruple of integral weighted coordinates, let \(D\) be the \((N-1)\times N\) integer matrix whose entry \(D_{ij}\) is the value at \(P_i\) of the \(j\)-th monomial of \cref{cor:support}(1), and assume \(\rank D=N-1\). Set \(B=\max_{i,j}|D_{ij}|\). Then
\[
h(G_{n^2})\;\leq\;(N-1)\,\log_2\bigl(B\sqrt{N-1}\bigr)+1.
\]
\end{lem}

\begin{proof}
The modular coefficient vector \(c\) corresponding to \(G_{n^2}\) lies in the kernel of \(D\), which is one-dimensional.
Let \(D^{(j)}\) be \(D\) with the \(j\)-th column deleted and let \(v\) be the integer vector with \(v_j=(-1)^{j}\det D^{(j)}\).
By Cramer's rule \(v\) lies in the kernel of \(D\) and is nonzero, so \(v\) spans it.
The vector \(c\) is primitive, so \(v=mc\) for some nonzero \(m\in\Z\), and \(|c_j|\leq|\det D^{(j)}|\) for every \(j\).
Each row of \(D^{(j)}\) has Euclidean norm at most \(B\sqrt{N-1}\), so Hadamard's inequality gives \(|\det D^{(j)}|\leq\bigl(B\sqrt{N-1}\bigr)^{N-1}\). The number of binary digits of an integer of absolute value at most \(X\) is at most \(\log_2 X+1\).
\end{proof}

\Cref{lem:hadamard} is far from sharp.
At \(n=5\) the points supplied by \cref{lem:n5-cover} with parameters of small height have \(\log_2 B\) of order \(1.1\cdot 10^{4}\), the largest entry being \(\chi_{10}^{48}\), so the bound is of order \(2.6\cdot 10^{8}\) bits, against the value \(h(G_{25})=1268\) recorded in \cref{tab:forms}.
The bound also depends on the points, not on \(n\) alone.

The modular height is thus known only where the exact Humbert form has been computed, and no bound on \(h(G_{n^2})\) in terms of \(n\) follows from the results above.  The size of the linear algebra is prescribed, while the number of primes required for reconstruction is not.  The proof of \cref{thm:G25-computation} does not require an a priori height bound or exhaustion of the sufficient prime count above.

Explicit dominant families are available for every \(n\leq 11\) \cite{Kum15}, so
\cref{prop:determination} applies throughout the range accessible to dense linear algebra. For
\(n\) without a known dominant family, the uniqueness of a truncated torsion system in the
sense of \cref{subsec:scheme} replaces the family.